\documentclass[10pt,psamsfonts]{amsart}
\usepackage{eucal}
\usepackage{amsmath,amssymb,amsfonts,amsthm,mathrsfs}
\usepackage{graphicx, array}
\usepackage{enumerate}
\usepackage{enumitem}
\usepackage{epsfig}
\usepackage{float}
\usepackage{color}
\usepackage{xcolor}
\colorlet{BLUE}{blue}
\usepackage[foot]{amsaddr}
\usepackage[T1]{fontenc}
\usepackage{enumitem}
\usepackage[colorlinks=true, urlcolor=blue, linkcolor=red]{hyperref}
\usepackage{mwe}
\usepackage{subcaption}
\usepackage[mathlines]{lineno}
\usepackage[hypcap=false]{caption}
\usepackage{listing}
\usepackage{hyperref}
\usepackage[encapsulated]{CJK}
\usepackage{caption}

\newtheorem{theorem}{Theorem}[section]
\newtheorem{lemma}{Lemma}[section]
\theoremstyle{corollary}

\newtheorem{proposition}{Proposition}[section]
\theoremstyle{definition}
\newtheorem{definition}{Definition}[section]
\newtheorem{example}{Example}[section]
\newtheorem{remark}{Remark}[section]
\numberwithin{equation}{section}

\def\longdelete#1{}

\begin{document}
\begin{CJK}{UTF8}{bsmi}
\title[Coexistence]{non-monotone traveling waves of the weak competition Lotka-Volterra system}

\author[Chen, C-C]{Chiun-Chuan Chen$^{1,3}$}
\address{$^1$ Department of Mathematics, National Taiwan University, Taiwan and National Center for Theoretical Sciences, Taiwan}
\email[Chen, C-C]{$^1$chchchen@math.ntu.edu.tw}

\author[Hsiao, T-Y]{Ting-Yang Hsiao$^2$}
\address{$^2$ Mathematics, Scuola Internazionale Superiore di Studi Avanzati (SISSA), Trieste, Italy}
\email[Hsiao, T-Y]{$^2$thsiao@sissa.it}

\author[Wang, S-C]{Shun-Chieh Wang$^3$}
\address{$^3$ (corresponding author) National Center for Theoretical Science, Taipei, Taiwan}
\email[Wang, S-C]{$^3$rjaywang1130@ncts.ntu.edu.tw}



\footnote{These authors contributed equally to this work.}

\date{\today}

\begin{abstract}
We investigate traveling wave solutions in the two-species reaction-diffusion Lotka--Volterra competition system under weak competition. For the strict weak competition regime $(b<a<1/c,\,d>0)$, we construct refined upper and lower solutions combined with the Schauder fixed point theorem to establish the existence of traveling waves for all wave speeds 
$s\geq s^*:=\max\{2,2\sqrt{ad}\}$, and provide verifiable sufficient conditions for the emergence of non-monotone waves. Such conditions for non-monotone waves have not been explicitly addressed in previous studies. It is interesting to point out that our result for non-monotone waves also holds for the critical speed case $s=s^*$. In addition, in the critical strong-weak competition case $(b<a=1/c,\,d>0)$, we rigorously prove, for the first time, the existence of front-pulse traveling waves. 
\end{abstract}

\maketitle

\section{Introduction}
In population biology, the Lotka--Volterra competition equations are widely accepted as a fundamental model for describing the interactions between competing species. By incorporating spatial diffusion into these equations, one arrives at the Lotka--Volterra competition-diffusion system, which provides a natural framework for studying the spatial propagation and coexistence of biological populations. Mathematically, such systems belong to the class of reaction-diffusion equations that admit traveling wave solutions, a central object in the study of spatial ecology and pattern formation.  

For the two-species case, the system can be written in the form
\begin{equation}\label{general eq}
\begin{cases}
u_t = u_{xx} + u(1-u-cv), \\
v_t = d v_{xx} + v(a-bu-v),
\end{cases}
\quad x \in \mathbb{R}, \ t>0,
\end{equation}
where $u(x,t)$ and $v(x,t)$ denote the population densities of two competing species, and $a,b,c,d>0$ are parameters reflecting the competition intensity and diffusion rate. Depending on the coefficients, the system admits several equilibria, among which the coexistence equilibrium plays a decisive role under the weak competition condition.

We consider the traveling wave ansatz for \eqref{general eq} of the form
\begin{equation}
    (u(x,t),v(x,t))=(u(\xi),v(\xi)), \quad \xi = x+st,
\end{equation}
where $s$ denotes the wave speed. We write $\partial_{\xi}$ as a prime. A direct calculation shows that $(u,v)$ satisfies
\begin{equation} \label{maineq}
\begin{cases}
    u'' - su' + u(1-u-cv) = 0, \\
    dv'' - sv' + v(a-bu-v) = 0,
\end{cases}
\quad \xi \in \mathbb{R}.
\end{equation}
A classical problem, first posed by Tang and Fife \cite{MP1980}, is to determine whether, in the strict weak competition regime
\begin{equation}\label{weak}
    b < a < \frac{1}{c},
\end{equation}
there exist traveling wave solutions $(u(\xi),v(\xi),s)$ connecting the extinction state $(0,0)$ to the coexistence equilibrium
\[
    (u^*,v^*)=\left(\frac{1-ac}{1-bc}, \,\frac{a-b}{1-bc}\right).
\]
Equivalently, Tang and Fife investigated solutions of \eqref{maineq} subject to the boundary conditions
\begin{equation} \label{asymptotic behavior}
    \lim_{\xi \to -\infty}(u,v)(\xi)=(0,0), \qquad \lim_{\xi \to +\infty}(u,v)(\xi)=(u^*,v^*).
\end{equation}
They proved that whenever the wave speed $s$ exceeds the critical threshold 
\[
    s^* := \max\{2,2\sqrt{ad}\},
\]
there exists a traveling wave with strictly monotone profiles and no positive wave satisfies (\ref{asymptotic behavior}) if $s<s^*$. This result provided the first rigorous demonstration that competitive interactions can generate spatial invasion dynamics governed by monotone wave fronts.  For the $n$ species related work, please see \cite{van1995existence}.

The weak competition regime has played one of the central roles in the study of traveling wave solutions for Lotka--Volterra competition-diffusion systems. Following the work of Tang and Fife, related research has made further progress. 
In particular, Ma \cite{S.Ma} introduced comparison principles and super-subsolution techniques that provided more flexible sufficient conditions for the existence of monotone traveling fronts under weak competition assumptions. 
Building on this framework, the work \cite{sonego2024control} designed a boundary control scheme driven by traveling waves, explicitly exploiting the monotonicity guaranteed in the weak competition setting. 
More recently, Chang and Wu \cite{chang2025bistable} extended this direction to three-species systems: by assuming weak competition between two species, they were able to preserve the monotonicity of two-species subsystems and consequently establish the existence of three-species traveling fronts. 
These contributions highlight that weak competition not only ensures coexistence equilibria but also provides the structural monotonicity necessary for rigorous analysis and further applications of traveling fronts.

By contrast, the study of non-monotone traveling waves remains comparatively limited. 
Hung \cite{hung2012exact} constructed exact traveling waves under specially chosen parameters, thereby demonstrating the existence of non-monotone solutions in certain cases. 
Lin and Ruan \cite{lin2014traveling} further observed the possibility of non-monotone waves and supported their existence by concrete examples and numerical simulations. 
Nevertheless, broad sufficient conditions for non-monotone fronts are still lacking, and fundamental issues, such as the construction of front-pulse type non-monotone solutions or the existence of non-monotone waves at the critical wave speed $s=s^*$, remain open. 
We also point out related developments: in \cite{chang2022series}, the authors investigated the stability of front-pulse solutions; in \cite{Yang}, Yang constructed front-pulse solutions for the case where one species dominates the other; and in \cite{hung2012exact}, exact front-pulse solutions were provided. 
These studies suggest that while monotone waves in weak competition have been better understood, a systematic description of non-monotone traveling waves is still far from complete.

We emphasize that the present work exclusively focuses on the weak competition regime. 
For the strong competition case, see \cite{Kanon} and references therein, where the existence of traveling waves was established and the wave speed was shown to depend analytically on the competition coefficients. 
In addition, Morita and Tachibana \cite{morita2009entire} constructed entire solutions in the strong competition setting, providing a broader dynamical picture of invasion phenomena beyond classical fronts.  
Further developments on spreading speeds and traveling fronts under strong competition can be found in 
\cite{JY,chang2023propagating,peng2021sharp,jiang2025spreading}, 
where qualitatively different behaviors such as bistability arise. For more details concerning the N-barrier maximum principle and its applications in reaction-diffusion systems, 
we refer to \cite{chen2016maximum,chen2016nonexistence,chen2016n,hung2016n,chen2020discrete,hsiao2022estimates}.

\subsection{Assumptions and notations}

Before stating our main results, we introduce several notions.  
The critical wave speed is defined by
\[
s^* := \max\{2,2\sqrt{ad}\}.
\]
We distinguish two parameter regimes:
\begin{itemize}
    \item \emph{Strict weak competition}: $b<a<\frac{1}{c}$, in which case the coexistence equilibrium $(u^*,v^*)$
    is strictly positive.
    \item \emph{Critical strong-weak} competition: $b<a=\frac{1}{c}$, where the equilibrium degenerates to $(0,v^*)$.
\end{itemize}
In the strict weak competition case, a front traveling wave refers to a solution of
\eqref{maineq} satisfying \eqref{asymptotic behavior}.  
In the critical strong-weak competition case, a front-pulse traveling wave is a solution of \eqref{maineq} satisfying the same boundary condition but with pulse-like behavior in one component.

\subsection{Main results}

We summarize our main theorems concerning the existence of traveling wave solutions to~\eqref{maineq} under the weak competition assumption.  
Our results cover both the strict weak competition regime $b<a<1/c$ and the critical strong-weak competition regime $b<a=1/c$.  
Depending on the wave speed $s$, we establish precise sufficient conditions for the existence of monotone and non-monotone fronts, as well as front-pulse solutions in the degenerate setting. First, we reprove the classic result of Tang and Fife via a different approach, i.e., the sub-sup solution method.

\begin{proposition}\label{maintheorem}  Given $d>0$. Assume $a,b,c$ satisfy $(\ref{weak})$. Then \eqref{maineq} and \eqref{asymptotic behavior} admit a traveling wave solution $(u,v)(\xi)$ if and only if $s \geq s^*$. 
\end{proposition}

Proposition~\ref{maintheorem} ensures the existence of a traveling wave solution, but does not address its monotonicity.  
We next turn to the refined question of when the profiles become non-monotone, and provide explicit sufficient conditions. Before that, we give some definitions.

\begin{definition} 
    We say a traveling wave solution $(u,v)(\xi)$ is monotone if and only if both $u$ and $v$ are monotone functions. We say a traveling wave solution $(u,v)(\xi)$ is non-monotone if and only if one of $u$ or $v$ is not a monotone function.
\end{definition}

\begin{definition} 
    We say $f(\xi):\mathbb{R} \to \mathbb{R}$ is eventually monotone at $+\infty$ (resp., $-\infty$) if there exists $M>0$ (resp., $-M<0$) such that $f(\xi)$ is a monotone function for all $\xi>M$ (resp., $x<-M$). We say $f(\xi)$ oscillates at $+\infty$ (resp., $-\infty$) if $f$ is not eventually monotone at $+\infty$ (resp., $-\infty$).
\end{definition}

We now turn to sufficient conditions for non-monotone solutions as follows.

\begin{theorem}\label{unonmonotone}
    For any $d>0$, $b<a$, $s \geq s^*$ there exists $\delta(a,b,s)>0$ such that under the condition $\delta(a,b,s)<c<\frac{1}{a}$, there exists a non-monotone traveling wave solution $(u,v)(\xi)$ satisfying $(\ref{maineq})$ and $(\ref{asymptotic behavior})$.
\end{theorem} 

\begin{theorem}\label{vnonmonotone}
    For any $d>0$, $a<\frac{1}{c}$, $s \geq s^*$ there exists $\delta(d,a,c,s)>0$ such that under the condition $\delta(d,a,c,s)<b<a$, there exists a non-monotone traveling wave solution $(u,v)(\xi)$ satisfying $(\ref{maineq})$ and $(\ref{asymptotic behavior})$.
\end{theorem} 

One can consider $\delta(a,b,s)$ in Theorem~\ref{unonmonotone} as a number close to but smaller than $1/a$ and $\delta(d,a,c,s)$ in Theorem~\ref{vnonmonotone} as a number close to but smaller than $a$. Theorems~\ref{unonmonotone} and~\ref{vnonmonotone} establish, for the strict weak competition regime, explicit sufficient conditions for the existence of non-monotone fronts when $s \geq s^*$.  
A natural next step is to investigate the borderline case $a=1/c$, namely the critical strong-weak competition regime.  
In this degenerate setting, the coexistence equilibrium reduces to a semi-trivial state, and the corresponding wave dynamics give rise to a new type of solution, which we call a front-pulse. Our final result characterizes the existence of such front-pulse solutions.

\begin{theorem} \label{nontrivialu}
    For any $b<a, \ ac=1, \ d>0$. If $s \geq s^*$, there exists a non-trivial front-pulse solution of $(\ref{degenerate1})$ with $u(\xi) \to 0$ as $|\xi| \to +\infty$.
\end{theorem}
    
We pause to remark that if we consider the other degenerate case, $b=a<\frac{1}{c}$, we can still obtain a front-pulse solution of $(\ref{degenerate2})$ with $v(\xi) \to 0$ as $|\xi| \to +\infty$. Please see Proposition \ref{v pulse}. Besides classical traveling waves, we note recent progress on the critical strong-weak competition case for the related system; see \cite{cai2025lotka, alfaro2023lotka}.

We summarize our main results as follows: 
\begin{enumerate} 
    \item For $b<a<\frac{1}{c},\ d>0$, we construct refined upper and lower solutions and employ the Schauder fixed point theorem to reprove the existence of traveling waves for all $s \geq s^*=\max\{2,2\sqrt{ad}\}$ in \cite{MP1980}. Moreover, we establish verifiable sufficient conditions for non-monotone waves, that were previously missing in the literature. We also confirm the nonexistence of traveling waves for $s<s^*$. For the case $s>s^*$ and delayed models, please see \cite{WGS2006}.

    \item In the threshold case $s=s^*$, we provide the first theoretical construction of non-monotone traveling waves, filling a long-standing gap where only numerical or example-based evidence had been available.

    \item For the degenerate case $b<a=\tfrac{1}{c},\ d>0$, the classical lower-solution approach of Lin-Ruan \cite{lin2014traveling}, which relies on the Fisher-KPP asymptotics, breaks down due to the critical competition balance $a=\frac{1}{c}$. To overcome this obstacle, we introduce a refined lower solution tailored to this regime, which enables the rigorous construction of a very interesting new class of front-pulse traveling waves.
\end{enumerate}

\subsection{Organization} 
We organize the paper in the following order. In Section \ref{Section 2}, we introduce the functional setting and establish the framework of sub- and super-solutions. 
Using Schauder's fixed point theorem, we prove the existence of traveling wave solutions for $s \geq s^*=\max\{2,2\sqrt{ad}\}$ via constructing explicit sub- and super-solutions both for $s>s^*$ and the critical case $s=s^*$ and apply the shrinking-box argument to show convergence to the coexistence equilibrium at $+\infty$. In Section \ref{Section 3}, we develop verifiable sufficient conditions for traveling waves: we first present conditions guaranteeing monotone fronts (Proposition \ref{P 1}), and then derive explicit, checkable sufficient conditions for non-monotone waves for both cases $s\ge s^*$, thereby proving Theorems \ref{unonmonotone} and \ref{vnonmonotone}. We also establish a cooperative oscillation property at $+\infty$: oscillation of one component forces oscillation of the other. In Section \ref{Section 4}, we address the critical strong-weak competition regime by analyzing two degenerate reductions of the original system. 
Using the non-monotone wave constructions from Section \ref{Section 3}, together with compactness and elliptic estimates, we prove Theorem \ref{nontrivialu}: the existence of non-trivial front-pulse traveling waves. 

\section{The existence of the Traveling wave solution} \label{Section 2}

\subsection{Preliminaries}
In this section, we give some preliminaries for our main purpose. First, we define our function space.
\begin{equation} \label{solutionspace}
    X=\{(u,v): \mathbb{R} \to \mathbb{R}^2 \text{ is continuous function }, \ 0 \leq u(\xi) \leq 1, 0 \leq ,v(\xi) \leq a \}.
\end{equation}

Also, according to the source term of $(\ref{maineq})$, we define the functions as $$ F_1 (u, v) =\beta u + u(1-u-cv),$$
and
$$F_2(u,v)=\beta v + v(a-bu-v),$$
for some constant $\beta >0$.  Then it is easy to see that for any $(u,v) \in X$, $F(u,v)$ is uniformly Lipschitz in $X$. We choose $\beta >0$ large enough so that $\frac{ \partial F_1}{\partial u} \geq 0, \ \frac{ \partial F_2}{\partial v} \geq 0$ in $X$. For this $\beta>0$, we can rewrite equation $(\ref{maineq})$ in 
\begin{equation}\label{system}
    \begin{pmatrix}
        d_1 & 0\\
        0 & d_2
    \end{pmatrix}\begin{pmatrix}
        u'' \\
        v''
    \end{pmatrix}-s\begin{pmatrix}
        u' \\
        v'
    \end{pmatrix}-\beta \begin{pmatrix}
        u \\
        v
    \end{pmatrix}+\begin{pmatrix}
        F_1(u,v) \\
        F_2(u,v)
    \end{pmatrix}=\begin{pmatrix}
        0 \\
        0
    \end{pmatrix}
\end{equation}
where $d_1=1$, and $d_2=d$. Now we define $\lambda_{i1}<0<\lambda_{i2}$ to be the solution of the quadratic equation
$$d_i r^2-s r-\beta=0, \ i=1,2.$$
For given $(u,v) \in X$, we consider the operator $P=(P_1,P_2): X \to X$ defined as following
\begin{equation}\label{solutionmap}
    P_i(u,v)(\xi)=\frac{1}{d_i(\lambda_{i2}-\lambda_{i1})}\left[ \int_{-\infty}^\xi e^{\lambda_{i1}(\xi-s)}+\int_{\xi}^{+\infty} e^{\lambda_{i2}(\xi-s)} \right]F_i(u,v)(s)ds,
\end{equation}
for $i=1,2, \ \xi \in \mathbb{R}$. By the variation of constant formula it is easy to see that $P$ satisfies the equation $(\ref{system})$. Next, we give the definition of super-solution and sub-solution of $(\ref{maineq})$ as follows.

\begin{definition}
    The continuous functions $(\Bar{u},\Bar{v})$ and $(\underline{u},\underline{v})$ are called a pair of super and sub solutions of $(\ref{maineq})$ if
   \begin{equation}
      \begin{cases}\label{supsub}
        \Bar{u}''-s\Bar{u}'+\Bar{u}(1-\Bar{u}-c\underline{v}) \leq 0,\\
        \underline{u}''-s\underline{u}'+\underline{u}(1-\underline{u}-c\Bar{v}) \geq 0,\\
        d\Bar{v}''-s\Bar{v}'+\Bar{v}(a-b\underline{u}-\Bar{v}) \leq 0,\\
        d\underline{v}''-s\underline{v}'+\underline{v}(a-b\Bar{u}-\underline{v}) \geq 0,
    \end{cases} 
   \end{equation}
    
    for all $\xi \in \mathbb{R}\setminus D$ with $D=\{\xi_1,\xi_2,...,\xi_N\}$.
\end{definition}

Finally, we present the most important theorem of existence of the solution in the next subsection.

\subsection{The Existence Theorem} By a standard argument such as that in, e.g., \cite{S.Ma}, \cite{YLHGLin}, we have the following existence theorem for the system (\ref{maineq}). The key idea is to use the super-solution and sub-solution to construct a weight subspace such that $P=(P_1,P_2)$ is a continuous, compact self-map, and apply Schauder's fixed point theory. We omit the details here.

\begin{lemma} \label{existence}
    Let $s>0$. Suppose $(\ref{maineq})$ has a pair of positive super-solution and sub-solution in $X$ satisfying

\

    (1): $\Bar{u}(\xi) \geq \underline{u}(\xi)$, \ $\Bar{v}(\xi) \geq \underline{v}(\xi)$ for all $\xi \in \mathbb{R}$.

\

    (2): $\Bar{u}'(\xi -) \geq \Bar{u}'(\xi +)$, \ $\Bar{v}'(\xi -) \geq \Bar{v}'(\xi +)$, \ $\underline{u}'(\xi -) \leq \underline{u}'(\xi +)$, \ $\underline{v}'(\xi -) \leq \underline{v}'(\xi +)$, for all $\xi \in D,$ where
    $$\overline{u}'(\xi \pm):=\lim_{z \to \xi \pm}\overline{u}'(z), \ \underline{u}'(\xi \pm):=\lim_{z \to \xi \pm}\underline{u}'(z),$$
    $$\overline{v}'(\xi \pm):=\lim_{z \to \xi \pm}\overline{v}'(z), \ \underline{v}'(\xi \pm):=\lim_{z \to \xi \pm}\underline{v}'(z).$$

Then $(\ref{maineq})$ has a positive solution $(u,v)$ such that $\underline{u}(\xi) \leq u(\xi) \leq \Bar{u}(\xi)$, $\underline{v}(\xi) \leq v(\xi) \leq \Bar{v}(\xi)$ for all $\xi \in \mathbb{R}$.
\end{lemma}

\subsection{The super-solution and sub-solution of $s>s^*$} We shall derive the existence of the traveling wave solution for $s>s^*=\max\{2,2\sqrt{ad}\}$. First of all we fix $s>s^*$, and define the following positive constants.
\begin{align}\label{lambda}
    \lambda_1=\frac{s-\sqrt{s^2-4}}{2}, \ \lambda_2=\frac{s-\sqrt{s^2-4ad}}{2d}, \ \lambda_3=\frac{s+\sqrt{s^2-4}}{2}, \ \lambda_4=\frac{s+\sqrt{s^2-4ad}}{2d}.
\end{align}
In fact, $\lambda_1$ and $\lambda_3$ are the positive solutions of
$$x^2-sx+1=0.$$
The $\lambda_2$ and $\lambda_4$ are the positive solution of
$$dx^2-sx+a=0.$$
According to Lemma \ref{existence}, we need to construct a pair of super and sub solutions of $(\ref{maineq})$. 

First, given any $\lambda>0, \ \mu, q >1$, it is easy to check that the function 
\begin{equation}\label{generalsupsub}
    f(\xi)=e^{\lambda \xi}-q e^{\mu \lambda \xi}
\end{equation}
has a unique zero $\xi_0=\frac{-\log{q}}{(\mu-1)\lambda}<0$ and a unique maximum point at $\xi_M = \frac{-\log{q \mu }}{(\mu-1)\lambda}<\xi_0$. Moreover, we have
\begin{equation}\label{maxoff}
    ||f||_{\infty}=f(\xi_M)=(1-\frac{1}{\mu})(q\mu)^{\frac{-1}{\mu-1}}.
\end{equation}
Since $f$ is continuous on $\mathbb{R}$ and positive on $(-\infty,\xi_0)$, for any small $\delta>0$ there exists $\Tilde{\xi} \in (\xi_M,\xi_0)$ such that $f(\Tilde{\xi})=\delta$ with $f'(\Tilde{\xi})<0$.

Next, we move to the second step. We sequentially select the constants $\mu_1, \mu_2, q_1, q_2, \delta_1$ and $\delta_2$ based on the following $(A1)-(A3)$.
\begin{align}
    &(A1) \text{ Let } \mu_1 \in (1,  \min\{\frac{\lambda_3}{\lambda_1},\frac{\lambda_1+\lambda_2}{\lambda_1},2\}), \  \mu_2 \in (1, \min\{ \frac{\lambda_4}{\lambda_2}, \frac{\lambda_1+\lambda_2}{\lambda_2},2\}) \text{ are very close to }1.\notag\\
    &(A2) \text{ Let } \ q_1 > \max\{1, \frac{1+ac}{-(\mu_1\lambda_1)^2+s(\mu_1 \lambda_1)-1}\}, \ q_2 > \max\{1, \frac{a^2+ab}{-d(\mu_2\lambda_2)^2+s(\mu_2 \lambda_2)-a}\}.\notag\\
    &(A3) \text{ Pick } \delta_1>0 \text{ such that } 0<\delta_1<\min\{1-ac,||f_1||_{\infty}\}, \text{ where }f_1(\xi) \text{ is defined by 
 }\notag\\
    & \text{(\ref{generalsupsub}) in parameters }(\lambda_1,\mu_1,q_1). \text{ Pick } \delta_2>0 \text{ such that } 0<\delta_2<\min\{a-b,||f_2||_{\infty}\},\notag\\
    &\text{ where }f_2(\xi) \text{ is defined by 
 } f_2(\xi)=ae^{\lambda_2 \xi}-q_2e^{\mu_2\lambda_2 \xi}.\text{ Note that there exists }\notag\\
    &\xi_i \in (\xi^i_M,\xi^i_0)\text{ such that }f(\xi_i)=\delta_i, i=1,2. \notag
\end{align}

Finally, we introduce the functions $\overline{u}(\xi), \underline{u}(\xi),\overline{v}(\xi), \underline{v}(\xi)$ as follows:

\begin{center}
$\begin{array}{ll}
\overline{u}(\xi)         =    
\begin{cases}
1 \ & \text{ if } \xi \geq 0,\\
e^{\lambda_1 \xi} \ & \text{ if } \xi \leq 0,
\end{cases} & \\
\underline{u}(\xi)       =      
\begin{cases}
\delta_1 \ &  \text{ if } \xi \geq \xi_1,\\
e^{\lambda_1 \xi}-q_1e^{\mu_1\lambda_1 \xi} \ &   \text{ if } \xi \leq \xi_1,
\end{cases} & \\
\overline{v}(\xi)         =    
\begin{cases}
a \ & \text{ if } \xi \geq 0,\\
ae^{\lambda_2 \xi} \ & \text{ if } \xi \leq 0,
\end{cases} & \\
\underline{v}(\xi)       =      
\begin{cases}
\delta_2 \ &  \text{ if } \xi \geq \xi_2,\\
ae^{\lambda_2 \xi}-q_2e^{\mu_2\lambda_2 \xi} \ &   \text{ if } \xi \leq \xi_2,
\end{cases} & \\
\end{array}$
\end{center}
where $\xi_i<0$ is a point such that $\underline{u}$ and $\underline{v}$ are continuous functions on $\mathbb{R}$. It is easy to see that $(\overline{u},\underline{u},\overline{v},\underline{v})$ meets the assumption $(1)$ and $(2)$ in Lemma \ref{existence}. In fact, we have the following theorem.

\begin{lemma}\label{existofsupsub}
    For each $s > s^*$, there exists a positive solution $(u,v)(\xi)$ of $(\ref{maineq})$ with $\underline{u}(\xi) \leq u(\xi) \leq \overline{u}(\xi)$ and $\underline{v}(\xi) \leq v(\xi) \leq \overline{v}(\xi)$ for all $\xi \in \mathbb{R}$ such that 
    $$\lim_{\xi \to -\infty}(u,v)=(0,0).$$
\end{lemma}
\begin{proof}
    
By Lemma \ref{existence}, it is sufficient to check that $(\overline{u},\underline{u},\overline{v},\underline{v})$ satisfies the definition of sup-sub solutions. It is easy to check that $ \overline{u}(\xi), \underline{u}(\xi), \overline{v}(\xi),\underline{v}(\xi)$ satisfy the conditions of Lemma \ref{existence}. Therefore, we only need to check the differential inequalities. Without loss of generality, we check the differential inequalities of $\overline{u}(\xi)$ and $\underline{u}(\xi)$. The proofs of the other two differential inequalities are similar and are omitted for brevity. First, we claim that
    $$\overline{u}''(\xi)-s\overline{u}'(\xi)+\overline{u}(\xi)(1-\overline{u}(\xi)-c\underline{v}(\xi))\leq 0$$
    holds for $\xi \in \mathbb{R} \setminus \{0\}.$ For $\xi>0$, $\overline{u}(\xi)=1$ and 
    $$\overline{u}''(\xi)-s\overline{u}'(\xi)+\overline{u}(\xi)(1-\overline{u}(\xi)-c\underline{v}(\xi))=-c\underline{v}(\xi) \leq 0.$$
    When $\xi < 0$, $\overline{u}(\xi)=e^{\lambda_1 \xi}$ and
    \begin{align}
        &\overline{u}''(\xi)-s\overline{u}'(\xi)+\overline{u}(\xi)(1-\overline{u}(\xi)-c\underline{v}(\xi))\notag\\
        &=e^{\lambda_1 \xi}[\lambda_1^2-s\lambda_1+1]-\overline{u}^2(\xi)-c\overline{u}(\xi)\underline{v}(\xi)= -\overline{u}^2(\xi)-c\overline{u}(\xi)\underline{v}(\xi) \leq 0.\notag
    \end{align}
    Next, we claim that 
    $$\underline{u}''(\xi)-s\underline{u}'(\xi)+\underline{u}(\xi)(1-\underline{u}(\xi)-c\overline{v}(\xi))\geq 0$$
    holds for $\xi \in \mathbb{R} \setminus \{\xi_1\}.$ In the case $\xi > \xi_1$, we have $\underline{u}(\xi)=\delta_1$ and
    \begin{align}
        \underline{u}''(\xi)-s\underline{u}'(\xi)+\underline{u}(\xi)(1-\underline{u}(\xi)-c\overline{v}(\xi))\geq \delta_1 (1-\delta_1-ac) \geq 0.\notag
    \end{align}
    
    For $\xi < \xi_1$, we have $\underline{u}(\xi)=e^{\lambda_1 \xi}-q_1e^{\mu_1\lambda_1 \xi}$ and 
    \begin{align}
        &\underline{u}''(\xi)-s\underline{u}'(\xi)+\underline{u}(\xi)(1-\underline{u}(\xi)-c\overline{v}(\xi))\notag\\
        &=-q_1[(\mu_1\lambda_1)^2-s(\mu_1\lambda_1)+1]e^{\mu_1 \lambda_1 \xi}-(e^{\lambda_1 \xi}-q_1 e^{\mu_1 \lambda_1 \xi})^2-ac(e^{\lambda_1 \xi}-q_1 e^{\mu_1 \lambda_1 \xi}) e^{\lambda_2 \xi}\notag\\
        & \geq -q_1[(\mu_1\lambda_1)^2-s(\mu_1\lambda_1)+1]e^{\mu_1 \lambda_1 \xi}-e^{2 \lambda_1 \xi}-ac e^{(\lambda_1+\lambda_2) \xi}\notag\\
        &\geq e^{\mu_1 \lambda_1 \xi}[-q_1[(\mu_1\lambda_1)^2-s(\mu_1\lambda_1)+1]-e^{((2-\mu_1)\lambda_1) \xi}-ac e^{((\lambda_1+\lambda_2)-\mu_1 \lambda_1)\xi}]\notag\\
        &\geq e^{\mu_1 \lambda_1 \xi}[-q_1[(\mu_1\lambda_1)^2-s(\mu_1\lambda_1)+1]-1-ac]>0,\notag
    \end{align}
    where we use the definition of $\lambda_1, \mu_1$ and $q_1$ in $(A1)-(A3)$. 
    
    Finally, the limit
    \begin{equation}
        \lim_{\xi \to -\infty}(u,v)=(0,0),\notag
    \end{equation}
    can be proved by Squeeze Theorem. Therefore, the proof of this theorem has been complete.
\end{proof}

\subsection{The super-solution and sub-solution of $s=s^*$} We shall derive the existence of the traveling wave solution for $s=s^*\}$. First of all we set $s=s^*$. Since the system is symmetry we only need to consider the case $ad \leq 1$. Another case, with a similar construction, will be omitted here. Consider $ad \leq 1$, let $s=s^*$, define the following positive constants

\begin{align}\label{lambdahat}
    \hat{\lambda}_1=\hat{\lambda}_3=\frac{s}{2}, \ \hat{\lambda}_2=\frac{s-\sqrt{s^2-4ad}}{2d}, \  \hat{\lambda}_4=\frac{s+\sqrt{s^2-4ad}}{2d}.
\end{align}

Similar to previous section, given any $\lambda>0, h>0, \ q >1$, it is easy to check that the non-negative function 
\begin{equation}\label{generalsupsubmin}
    g(\xi)=(-h \xi-q \sqrt{-\xi})e^{\lambda \xi}, \ \ \xi \leq -(\frac{q}{h})^2
\end{equation}
has a unique zero $\hat{\xi}_0=-(\frac{q}{h})^2$ and a unique maximum point at $\hat{\xi}_M <\hat{\xi}_0$. Since $g$ is continuous and positive on $(-\infty,\hat{\xi}_0)$, for any small $\hat{\delta}>0$ there exists $\hat{\xi} \in (\hat{\xi}_M,\hat{\xi}_0)$ such that $g(\hat{\xi})=\hat{\delta}$ with $g'(\hat{\xi})<0$. Note that the unique zero $\hat{\xi}_0 \to -\infty$ when $q \to +\infty$.

We will divide the construction into two cases. 

Case 1. $ad=1$.

We sequentially select the constants $\hat{h}_1, \hat{h}_2, \hat{q}_1, \hat{q}_2, \hat{\delta}_1$ and $\hat{\delta}_2$ based on the following $(B1)-(B4)$.

\begin{align}
&(B1) \text{ Let } \hat{h}_1=\frac{\hat{\lambda}_1}{\hat{\lambda}_1+1}e^{\hat{\lambda}_1 +1}, \ \hat{h}_2=\frac{a\hat{\lambda}_2}{\hat{\lambda}_2+1}e^{\hat{\lambda}_2 +1}.\notag\\
    &(B2) \text{ Let } \hat{q}_1 > \max \left\{ \sqrt{\hat{h}_1(\frac{1}{\hat{\lambda}_1}+1)}, \ 4\left( c  \hat{h}_1 \hat{h}_2 (\frac{7}{2e\hat{\lambda}_2})^{\frac{7}{2}}  + \hat{h}_1^2(\frac{7}{2e\hat{\lambda}_1})^{\frac{7}{2}} \right) \right\}. \notag\\
    &(B3) \text{ Let } \hat{q}_2 > \max \left\{\sqrt{\hat{h}_2(\frac{1}{\hat{\lambda}_2}+1)}, \ \frac{4}{d}\left( b  \hat{h}_1 \hat{h}_2 (\frac{7}{2e\hat{\lambda}_1})^{\frac{7}{2}}  + \hat{h}_2^2(\frac{7}{2e\hat{\lambda}_2})^{\frac{7}{2}} \right) \right\}. \notag\\
    &(B4) \text{ Pick } \hat{\delta}_1 >0 \text{ so small such that  } 0<\hat{\delta}_1<\min\{1-ac,||g_1||_{\infty}\}, \text{ where }g_1(\xi) \text{ is defined by 
 }\notag\\
    & \text{(\ref{generalsupsubmin})}. \text{ Pick } \hat{\delta}_2>0 \text{ so small such that  } 0<\hat{\delta}_2<\min\{a-b,||g_2||_{\infty}\}, \text{ where }g_2(\xi) \text{ is defined by 
 }\notag\\
    & \text{(\ref{generalsupsubmin})}.\text{ Note that there exists } \hat{\xi}_i \in (\hat{\xi}^i_M,\hat{\xi}^i_0)\text{ such that }g_i(\hat{\xi}_i)=\hat{\delta}_i, i=1,2. \notag
\end{align}

In this case, we introduce the following functions $\overline{u}(\xi), \underline{u}(\xi),\overline{v}(\xi), \underline{v}(\xi)$ as follows.

\begin{center}
$\begin{array}{ll}
\overline{u}(\xi)         =    
\begin{cases}
1 \ & \text{ if } \xi \geq \frac{-1}{\hat{\lambda}_1}-1,\\
-\hat{h}_1\xi e^{\hat{\lambda}_1 \xi} \ & \text{ if } \xi \leq \frac{-1}{\hat{\lambda}_1}-1,
\end{cases} & \\
\underline{u}(\xi)       =      
\begin{cases}
\hat{\delta}_1 \ &  \text{ if } \xi \geq \hat{\xi_1},\\
[-\hat{h}_1\xi -\hat{q}_1 \sqrt{-\xi}]e^{\hat{\lambda}_1 \xi} \ &   \text{ if } \xi \leq \hat{\xi_1},
\end{cases} & \\
\overline{v}(\xi)         =    
\begin{cases}
a \ & \text{ if } \xi \geq \frac{-1}{\hat{\lambda}_2}-1,\\
-\hat{h}_2\xi e^{\hat{\lambda}_2 \xi} \ & \text{ if } \xi \leq \frac{-1}{\hat{\lambda}_2}-1,
\end{cases} & \\
\underline{v}(\xi)       =      
\begin{cases}
\hat{\delta}_2 \ &  \text{ if } \xi \geq \hat{\xi_2},\\
[-\hat{h}_2\xi -\hat{q}_2 \sqrt{-\xi}]e^{\hat{\lambda}_2 \xi} \ &   \text{ if } \xi \leq \hat{\xi_2},
\end{cases} & \\
\end{array}$
\end{center}
where $\hat{\xi_i}<-(\frac{\hat{q}_i}{\hat{h}_i})^2<\frac{-1}{\lambda_i}-1<0$ is a point such that $\underline{u}$ and $\underline{v}$ are continuous functions on $\mathbb{R}$ and $\hat{h}_1=\frac{\lambda_1}{\lambda_1+1}e^{\lambda_1 +1}, \ \hat{h}_2=\frac{a\lambda_2}{\lambda_2+1}e^{\lambda_2 +1}$. It is easy to see that $(\overline{u},\underline{u},\overline{v},\underline{v})$ meets the assumption $(1)$ and $(2)$ in Lemma \ref{existence}. 

Case 2. $ad<1$.

We select the constants $\hat{h}_1, \hat{q}_1, \hat{\delta}_1, \hat{\xi}_1$ in a similar way to the above and choose $\hat{Q}_2, \hat{\delta}_2 ,\hat{\mu}_2$ and $\hat{\xi}_2$ according to the following $(B5)-(B7)$. 

\begin{align}
    &(B5) \text{ Let } \hat{\mu}_2 \in (1,  \min\{\frac{\lambda_4}{\lambda_2},1+\frac{\hat{\lambda}_1}{2 \lambda_2},2\}), \text{ is very close to }1.\notag\\
    &(B6) \text{ Let } \ \hat{Q}_2 > \max\{1, \frac{a^2+\frac{2ab\hat{h}_1e^{-1}}{\hat{\lambda}_1}}{-d(\hat{\mu}_2\lambda_2)^2+s(\hat{\mu}_2 \lambda_2)-a}\}.\notag\\
    &(B7) \text{ Pick } \hat{\delta}_2>0 \text{ such that } 0<\hat{\delta}_2<\min\{a-b,||f_2||_{\infty}\},\notag\\
    &\text{ where } \lambda_4, \lambda_2 \text{ are defined by }(\ref{lambda})  \text{ and } f_2(\xi) \text{ is defined by 
 }f_2(\xi)=ae^{\lambda_2 \xi}-\hat{Q}_2e^{\hat{\mu}_2\lambda_2 \xi}. \notag\\
 &\text{ And choose }\hat{\xi}_2 \text{ such that } \underline{v}(\xi) \text{ be }C^0(\mathbb{R}).\notag
\end{align}

In this case, we consider the functions $\overline{u}(\xi), \underline{u}(\xi),\overline{v}(\xi), \underline{v}(\xi)$ as follows.

\begin{center}
$\begin{array}{ll}
\overline{u}(\xi)         =    
\begin{cases}
1 \ & \text{ if } \xi \geq \frac{-1}{\hat{\lambda}_1}-1,\\
-\hat{h}_1\xi e^{\hat{\lambda}_1 \xi} \ & \text{ if } \xi \leq \frac{-1}{\hat{\lambda}_1}-1,
\end{cases} & \\
\underline{u}(\xi)       =      
\begin{cases}
\hat{\delta}_1 \ &  \text{ if } \xi \geq \hat{\xi_1},\\
[-\hat{h}_1\xi -\hat{q}_1 \sqrt{-\xi}]e^{\hat{\lambda}_1 \xi} \ &   \text{ if } \xi \leq \hat{\xi_1},
\end{cases} & \\
\overline{v}(\xi)       =  
\begin{cases}
a \ & \text{ if } \xi \geq 0,\\
ae^{\lambda_2 \xi} \ & \text{ if } \xi \leq 0,
\end{cases} & \\
\underline{v}(\xi)       =      
\begin{cases}
\hat{\delta}_2 \ &  \text{ if } \xi \geq \hat{\xi}_2,\\
ae^{\lambda_2 \xi}-\hat{Q}_2e^{\hat{\mu}_2 \lambda_2 \xi} \ &   \text{ if } \xi \leq \hat{\xi}_2,
\end{cases} & \\
\end{array}$
\end{center}
where $\lambda_2$ is defined in $(\ref{lambda})$. It is easy to see that $(\overline{u},\underline{u},\overline{v},\underline{v})$ satisfying the assumption $(1)$ and $(2)$ in Lemma \ref{existence}.  Combining these two cases, we obtain the following theorem.
\begin{lemma}\label{existofsupsubmin}
    For $s = s^*$, there exists a positive solution $(u,v)(\xi)$ of $(\ref{maineq})$ with $\underline{u}(\xi) \leq u(\xi) \leq \overline{u}(\xi)$ and $\underline{v}(\xi) \leq v(\xi) \leq \overline{v}(\xi)$ for all $\xi \in \mathbb{R}$ such that
    $$\lim_{\xi \to -\infty}(u,v)=(0,0).$$
\end{lemma}

\begin{proof}

    Consider the case $ad=1$. By Lemma \ref{existence}, it is sufficient to check that $(\overline{u},\underline{u},\overline{v},\underline{v})$ satisfy the definition of sup-sub solutions. Moreover, we only need to check the differential inequalities. It is easy to see that  $0< \overline{u}(\xi)<\underline{u}(\xi)\leq 1, \ 0< \overline{v}(\xi)<\underline{v}(\xi) \leq a$ for all $\xi \in \mathbb{R}$. Without loss of generality, we check the differential inequalities of $\overline{u}(\xi)$ and $\underline{u}(\xi)$. The other two differential inequalities can be proved in a similar way. First, we claim that
    $$\overline{u}''(\xi)-s\overline{u}'(\xi)+\overline{u}(\xi)(1-\overline{u}(\xi)-c\underline{v}(\xi))\leq 0$$
    holds for $\xi \in \mathbb{R} \setminus \{\frac{-1}{\lambda_1}-1\}.$ For $\xi>\frac{-1}{\lambda_1} -1$, $\overline{u}(\xi)=1$ and 
    $$\overline{u}''(\xi)-s\overline{u}'(\xi)+\overline{u}(\xi)(1-\overline{u}(\xi)-c\underline{v}(\xi))=-c\underline{v}(\xi) \leq 0.$$
    When $\xi < \frac{-1}{\lambda_1} -1$, $\Bar{u}(\xi)=-\hat{h}_1\xi e^{\hat{\lambda}_1 \xi}$ and 
        $$\overline{u}''(\xi)-s\overline{u}'(\xi)+\overline{u}(\xi)(1-\overline{u}(\xi)-c\underline{v}(\xi))=-c\Bar{u}(\xi)\underline{v}(\xi) \leq 0.$$
Next, we prove that 
    $$\underline{u}''(\xi)-s\underline{u}'(\xi)+\underline{u}(\xi)(1-\underline{u}(\xi)-c\overline{v}(\xi))\geq 0$$
    holds for $\xi \in \mathbb{R} \setminus \{\hat{\xi}_1\}.$ For $\xi > \hat{\xi}_1$, the inequality holds by a similar argument as in the case $s>s^*$. When $\xi < \hat{\xi}_1$, we have $\underline{u}(\xi)=[-\hat{h}_1\xi -\hat{q}_1 \sqrt{-\xi}]e^{\hat{\lambda}_1 \xi}$ and 
    \begin{align}
        &\underline{u}''(\xi)-s\underline{u}'(\xi)+\underline{u}(\xi)(1-\underline{u}(\xi)-c\overline{v}(\xi))\notag\\
=&-s(\frac{\hat{q}_1}{2}(-\xi)^{\frac{-1}{2}})e^{\hat{\lambda}_1 \xi}+\frac{\hat{q}_1}{4}(-\xi)^{\frac{-3}{2}}e^{\hat{\lambda}_1 \xi}+\hat{\lambda}_1\hat{q}_1(-\xi)^{\frac{-1}{2}}e^{\hat{\lambda}_1 \xi}\notag\\
&-c\Bar{v}(\xi)(-\hat{h}_1\xi-\hat{q}_1(-\xi)^{\frac{1}{2}})e^{\hat{\lambda}_1 \xi}-(-\hat{h}_1\xi-\hat{q}_1(-\xi)^{\frac{1}{2}})^2 e^{2 \hat{\lambda}_1 \xi}\notag  \\
=&e^{ \hat{\lambda}_1 \xi}\left[ -\frac{s \hat{q}_1}{2}(-\xi)^{\frac{-1}{2}}+\frac{\hat{q}_1}{4}(-\xi)^{\frac{-3}{2}}+\hat{\lambda}_1 \hat{q}_1(-\xi)^{\frac{-1}{2}}-c \Bar{v}(\xi) (-\hat{h}_1\xi-\hat{q}_1(-\xi)^{\frac{1}{2}})   \right]\notag\\
&-(-\hat{h}_1\xi-\hat{q}_1(-\xi)^{\frac{1}{2}})^2 e^{2 \hat{\lambda}_1 \xi}\notag\\
\geq& e^{ \hat{\lambda}_1 \xi}\left[ 
 \frac{\hat{q}_1}{4}(-\xi)^{\frac{-3}{2}}-c  \hat{h}_1 \hat{h}_2 |\xi|^2e^{\hat{\lambda}_2 \xi} - \hat{h}_1^2|\xi|^{2} e^{ \hat{\lambda}_1 \xi} \right] \notag\\
 \geq& (-\xi)^{\frac{-3}{2}}e^{ \hat{\lambda}_1 \xi}\left[ 
 \frac{\hat{q}_1}{4}-c  \hat{h}_1 \hat{h}_2 (\frac{7}{2e\hat{\lambda}_2})^{\frac{7}{2}}  - \hat{h}_1^2(\frac{7}{2e\hat{\lambda}_1})^{\frac{7}{2}} \right] \geq 0, \notag
\end{align}
    where we use the definition of $\hat{\lambda}_1,$ and the fact that given $\lambda>0$,
    $$(-\xi)^{\frac{7}{2}}e^{\lambda \xi} \leq (\frac{7}{2e\lambda})^{\frac{7}{2}} \text{ for }\xi \leq 0.$$

The proof of case $ad<1$ is very similar to the proof of case $s>s^*$, we only check
$$\underline{v}''(\xi)-s\underline{v}'(\xi)+\underline{v}(\xi)(1-\underline{v}(\xi)-c\overline{v}(\xi))\geq 0$$
here. For $\xi > \hat{\xi}_2$, the inequality holds by a similar argument as in the case $s>s^*$. When $\xi < \hat{\xi}_2$, we have

\begin{align}
    &d\underline{v}''(\xi)-s\underline{v}'(\xi)+\underline{v}(\xi)(a-b\overline{u}(\xi)-\underline{v}(\xi))\notag\\
    &=-\hat{Q}_2[d(\hat{\mu}_2 \lambda_2)^2-s(\hat{\mu}_2 \lambda_2)+a]e^{\hat{\mu}_2 \lambda_2 \xi}-b(-\hat{h}_1 \xi e^{\hat{\lambda}_1 \xi})(ae^{\lambda_2 \xi}-\hat{Q}_2e^{\hat{\mu}_2 \lambda_2 \xi})-(ae^{\lambda_2 \xi}-\hat{Q}_2e^{\hat{\mu}_2 \lambda_2 \xi})^2\notag\\
    &\geq -\hat{Q}_2[d(\hat{\mu}_2 \lambda_2)^2-s(\hat{\mu}_2 \lambda_2)+a]e^{\hat{\mu}_2 \lambda_2 \xi}-b(-\hat{h}_1 \xi e^{\hat{\lambda}_1 \xi})(ae^{\lambda_2 \xi})-a^2 e^{2\lambda_2 \xi} \notag\\
    & \geq e^{\hat{\mu}_2 \lambda_2 \xi}\left[-\hat{Q}_2[d(\hat{\mu}_2 \lambda_2)^2-s(\hat{\mu}_2 \lambda_2)+1]+ab\hat{h}_1(\xi e^{\frac{\hat{\lambda}_1}{2} \xi})e^{(\frac{\hat{\lambda}_1}{2}+(1-\hat{\mu}_2)\lambda_2) \xi}-a^2 e^{(2-\hat{\mu}_2)\lambda_2 \xi}\right]\notag\\
    &\geq e^{\hat{\mu}_2 \lambda_2 \xi}\left[-\hat{Q}_2[d(\hat{\mu}_2 \lambda_2)^2-s(\hat{\mu}_2 \lambda_2)+1]+ab\hat{h}_1(\frac{-2e^{-1}}{\hat{\lambda}_1})-a^2\right] \geq 0.\notag
\end{align}

Here we use the fact that 
$$\xi e^{\frac{\alpha}{2} \xi} \geq \frac{-2 e^{-1}}{\alpha} \text { for all }\xi \leq 0,$$
where $\alpha>0$. 

Finally, the limit
    \begin{equation}
        \lim_{\xi \to -\infty}(u,v)=(0,0),\notag
    \end{equation}
    can be proved by the Squeeze Theorem. Therefore, the proof of this theorem is complete.
\end{proof}
\subsection{The shrinking box argument} To analyze the asymptotic behavior of the positive solution obtained in Lemma \ref{existofsupsub} and Lemma \ref{existofsupsubmin} at $\xi \to +\infty$. We introduce the shrinking-box argument. This method can be referred to, for instance \cite{WGS2006}, \cite{gg2022}, \cite{YLHGLin}, and \cite{CTS}. We define the functions $m_u(\theta), m_v(\theta)$ and $M_u(\theta), M_v(\theta)$ for $\theta \in [0,1]$ as follows:
\begin{align}
    &m_u(\theta)=\theta u^*, \ M_u(\theta)=\theta u^*+(1-\theta)(1+\varepsilon),\\
    &m_v(\theta)=\theta v^*, \ M_v(\theta)=\theta v^*+(1-\theta)(a+\varepsilon),
\end{align}
where $\varepsilon$ is small enough such that $0<\varepsilon<\min\{\frac{1-ac}{c}, \frac{a-b}{b}\}$. For $0 < \theta_1 < \theta_2 <1$, it is easy to see that
\begin{align}
    &0=m_u(0)<m_u(\theta_1)<m_u(\theta_2)<m_u(1)=u^*=M_u(1)<M_u(\theta_2)<M_u(\theta_1)<M_u(0)=1+\varepsilon,\notag\\
    &0=m_v(0)<m_v(\theta_1)<m_v(\theta_2)<m_v(1)=v^*=M_v(1)<M_v(\theta_2)<M_v(\theta_1)<M_v(0)=a+\varepsilon.\notag
\end{align}
We are ready to show the tail behavior of the traveling wave solution at $+\infty$ as follows.
\begin{lemma}\label{shrinking}
    Let $(u,v)(\xi)$ be a positive solution obtained in Lemma \ref{existofsupsub} or Lemma \ref{existofsupsubmin}. Then 
\begin{align}\label{positiveinfty}
    \lim_{\xi \to +\infty}(u,v)(\xi)=(u^*,v^*).
\end{align}
\end{lemma}

\begin{proof}
By the fact that $\delta_1 \leq \underline{u}(\xi) \leq u(\xi) \leq \overline{u}(\xi) \leq 1$ and $\delta_2 \leq \underline{v}(\xi) \leq v(\xi) \leq \overline{v}(\xi) \leq a$ for all $\xi >0$, we obtain that
\begin{align}
    &\limsup_{\xi \to +\infty} u(\xi) \leq 1, \  \limsup_{\xi \to +\infty} v(\xi) \leq a\notag\\
    &\liminf_{\xi \to +\infty} u(\xi) \geq \delta_1, \  \liminf_{\xi \to +\infty} v(\xi) \geq \delta_2.\notag
\end{align}
Now we denote
\begin{align}
    &u^-=\displaystyle\liminf_{\xi \to +\infty}u(\xi), \ u^+=\displaystyle\limsup_{\xi \to +\infty}u(\xi)\notag\\
    &v^-=\displaystyle\liminf_{\xi \to +\infty}v(\xi), \ v^+=\displaystyle\limsup_{\xi \to +\infty}v(\xi).\notag
\end{align}
It is easy to see that 
\begin{align}
    & m_u(0)=0<u^- \leq u^+ <1+\varepsilon=M_u(0)\notag\\ 
    & m_v(0)=0<v^- \leq v^+ <a+\varepsilon=M_v(0).\notag
\end{align}
Note that (\ref{positiveinfty}) holds if we can show that 
\begin{equation}\label{shrink1}
    m_u(\theta)<u^- \leq u^+ <M_u(\theta) \ \text{ and } \ m_v(\theta)<v^- \leq v^+ <M_v(\theta)
\end{equation}
for all $\theta \in [0,1)$. Set $\theta_0:=\sup \{\theta \in [0,1): \text{ (\ref{shrink1}) holds}\}.$ Then $\theta_0$ is well-defined and it suffices to claim that $\theta_0=1$. Suppose $\theta_0<1$. Then passing to the limit, we have
\begin{align}
    & m_u(\theta_0) \leq u^- \leq u^+ \leq M_u(\theta_0),\notag\\ 
    & m_v(\theta_0)\leq v^- \leq v^+ \leq M_v(\theta_0).\notag
\end{align}
Moreover, by the definition of $\theta_0$, at least one of the following conditions holds:
\begin{align}
    u^-=m_u(\theta_0),  u^+ = M_u(\theta_0), v^-=m_v(\theta_0), v^+ = M_v(\theta_0).\notag
\end{align}
We only prove the first case; the proofs for the others are similar. Assume $u^-=m_u(\theta_0)$. We known that $v^+\leq M_v(\theta_0)$. If $u(\xi)$ is eventually monotone, then $u(+\infty)$ exist by $u(\xi)$ is bounded on $\mathbb{R}$. By the property of $\liminf$, we have
\begin{align}\label{shrinking 1}
    \liminf_{\xi \to +\infty} [1-u(\xi)-cv(\xi)]& \geq [1-\theta_0 u^*-\theta_0cv^*-(1-\theta_0)c(a+\varepsilon)]\notag\\
    &=(1-\theta_0)(1-ca-c\varepsilon)>0.
\end{align}
On the other hand, we have $\int_0^\infty u'(s)ds=u(+\infty)-u(0)$ is finite. Then we choose $\{\xi_n\}, \ \xi_n \to +\infty$ such that $\displaystyle\lim_{n \to +\infty}u(\xi_n)=m_u(\theta_0)$ and $\displaystyle\lim_{n \to +\infty}u'(\xi_n)=0$. Integrating the first equation of the system (\ref{maineq}) from $0$ to $\xi_n$, we obtain that
\begin{equation}\label{int0toxi_nu-}
    u'(\xi_n)-u'(0)-s[u(\xi_n)-u(0)]=-\int_0^{\xi_n} u(s)[1-u(s)-cv(s)]ds.
\end{equation}
Letting $n \to +\infty$, we get a contradiction, since the left-hand side of (\ref{int0toxi_nu-}) remains bounded and the right-hand side of (\ref{int0toxi_nu-}) tends to $-\infty$.

If $u(\xi)$ is oscillatory at $+\infty$, let $\xi_n \to +\infty$ be all the local minimum points of $u(\xi)$. It is easy to prove that $\displaystyle\liminf_{\xi \to \infty} u(\xi)=\displaystyle\liminf_{n \to \infty}u(\xi_n)=m_u(\theta_0)$. By taking the subsequence, we can choose $\{\xi_{n_k}\}$ such that $\displaystyle\lim_{k \to +\infty}u(\xi_{n_k})=m_u(\theta_0)$. Note that $u''(\xi_{n_k})-su'(\xi_{n_k}) \geq 0$ for all $k$. By $(\ref{shrinking 1})$, we obtain that
\begin{align}
    \liminf_{k \to +\infty} [u''(\xi_{n_k})-su'(\xi_{n_k})+u(\xi_{n_k})(1-u(\xi_{n_k})-cv(\xi_{n_k}))] >0,\notag
\end{align}
a contradiction. Hence $u^-=m_u(\theta_0)$ is impossible. Using similar argument as above we can still arrive other cases are impossible. Consequently, we must have $\theta_0=1$ and (\ref{positiveinfty}) follows.
\end{proof}

The proof of Proposition $\ref{maintheorem}$ is complete.

\subsection{Determination of the minimum speed} In this subsection, we would like to show that there is no positive solution of $(\ref{maineq})$ that satisfies $(\ref{asymptotic behavior})$ for $s<s^*.$ We have the following theorem.


\begin{lemma}\label{nonexist}
    Under the strict or critical competition cases, for $s<s^*$ there exists no positive solution to $(\ref{maineq})$ with the boundary condition 
    \[
        \lim_{\xi \to -\infty}(u,v)(\xi) = (0,0).
    \]
\end{lemma}

\begin{proof}
    Without loss of generality, we assume $s^*=2$. We first claim that $s \leq 0$ does not have a positive solution. Assume for some $\Bar{s} \leq 0$ there exists a positive solution of $(\ref{maineq})$ and $(\ref{asymptotic behavior})$. Then it follows from $(\ref{asymptotic behavior})$ that $\Psi(\xi):=1-u(\xi)-cv(\xi) \to 1$ as $\xi \to -\infty$. Hence there is a large $K>0$ such that
    $$\Psi(\xi) \geq \frac{1}{2}, \ \text{ for all }\xi \leq -K.$$
    If $u(\xi)$ oscillates near $-\infty$, then $u'(-\infty)=0$. An integration of the $u$-equation in $(\ref{maineq})$ from $-\infty$ to $\xi \leq -K$ gives
    \begin{align}
        0&=u'(\xi)-\Bar{s} u(\xi)+\displaystyle\int_{-\infty}^\xi u(\eta)(1-u(\eta)-cv(\eta))d\eta\geq u'(\xi)+\frac{1}{2}\displaystyle\int_{-\infty}^\xi u(\eta)d\eta.\notag
    \end{align}
    This implies, by an integration from $-\infty$ to $-K$ and using $u(-\infty) = 0$, that
    \begin{align}
        u(-K)=\displaystyle\int_{-\infty}^{-K}u'(\eta)d\eta \leq \frac{-1}{2}\int_{-\infty}^{-K}\int_{-\infty}^{\xi}u(\eta)d\eta d\xi < 0,\notag
    \end{align}
    a contradiction to the positivity of $u$ in $\mathbb{R}$. If $u(\xi)$ oscillates near $-\infty$, pick $\xi_n \to -\infty$ be the local minimum points of $u(\xi)$ with $u'(\xi_n)=0$ for all $n$. Apply the similar argument as above, and we also have $u(-K)<0$, a contradiction. Hence, we must have $s > 0$.

    Next, we can use the standard Sturm-Liouville argument to claim that for any $s \in (0,2)$ the equation $(\ref{maineq})$ and $(\ref{asymptotic behavior})$ has no positive solution. We left the proof in the Appendix \ref{Appendix A}. Therefore, the proof of this theorem is complete.    
\end{proof}
For the monotone solution, we quote the interesting results obtained by Tang and Fife \cite{MP1980} and Ma \cite{S.Ma}. For the non-monotone solution, we describe our construction with details.

\section{The profile of the solution} \label{Section 3}

\subsection{The monotone solution} Before constructing non-monotone solutions, we first recall some known results on existence and monotonicity of traveling wave solutions. For the existence of the monotonic solution, we have the following proposition which is due to Ma \cite[Theorem 2.1]{S.Ma}.
\begin{proposition}
    If $(\ref{maineq})$ has a non-constant super-solution $(\Bar{u}(\xi),\Bar{v}(\xi))$ and sub-solution $(\underline{u}(\xi),\underline{v}(\xi))$ satisfy
\begin{enumerate}
    \item $0 \leq \underline{u}(\xi) \leq \Bar{u}(\xi) \leq u^*$ and $0 \leq \underline{v}(\xi) \leq \Bar{v}(\xi) \leq v^*$ for all $\xi \in \mathbb{R}$;

    \item $\displaystyle\sup_{t \leq \xi}\underline{u}(t) \leq \Bar{u}(\xi)$, \ $\displaystyle\sup_{t \leq \xi}\underline{v}(t) \leq \Bar{v}(\xi)$ for all $\xi \in \mathbb{R}$;

    \item There is no constant equilibrium in the product set
    $$[(0,\displaystyle\inf_{\xi \in \mathbb{R}}\Bar{u}(\xi)] \cup [\sup_{\xi \in \mathbb{R}}\underline{u}(\xi),u^*)]\times[(0,\displaystyle\inf_{\xi \in \mathbb{R}}\Bar{v}(\xi)] \cup [\sup_{\xi \in \mathbb{R}}\underline{v}(\xi),v^*)].$$
    Then $(\ref{maineq})$ and $(\ref{asymptotic behavior})$ have a monotone solution. That is, $(\ref{maineq})$ has a traveling wavefront solution. 
\end{enumerate}
\end{proposition}

Here, we make a short common. If we know that any traveling wave solution satisfies $0<u(\xi)<u^*$ and $0<v(\xi)<v^*$ for all $\xi \in \mathbb{R}$, then Lemma 4.1 in \cite{wu2013spreading} shows that $(u,v)(\xi)$ is strictly monotone traveling wave function.

\begin{proposition}  \cite[Lemma 4.1]{wu2013spreading}
    \label{P 1}

For any $s \in \mathbb{R}$, if the wave profile $(u,v)$ satisfies \eqref{maineq} with $0 < u(x) < u^*$ and $0 < v(x) < v^*$ for any $x \in \mathbb{R}$, then $(u,v)$ is strictly monotone.
\end{proposition}

\color{black}
\subsection{The non-monotone solution}  We are interested to see whether there exists a non-monotone solution. In fact, the increase or decrease of the solution is very important. Knowing whether the solution is increasing or decreasing can lead to a better understanding of certain phenomena. If the solution $(u,v)(\xi)$ is strictly monotone, then the linearized operator of $(u,v)(\xi)$ has a positive kernel. If the solution is non-monotone, the corresponding linearized operator has a sign-changing kernel. Therefore, this type of solution is particularly interesting and more challenging to analyze. We provide several references on non-monotone solutions, namely \cite{marion2023existence}, and \cite{hung2012exact}. According to Proposition \ref{P 1}, we must find solutions whose range is not contained in $(0,u^*)$ or $(0,v^*)$. We give a special example to show that there exists a non-monotone traveling wave solution of $(\ref{maineq})$ and satisfies $(\ref{asymptotic behavior})$. We note that the existence theorem in Proposition \ref{maintheorem} has yet to express the monotonicity of the solution. Let us prove Theorems \ref{unonmonotone} and \ref{vnonmonotone}.


\begin{proof}[Proof of Theorem \ref{unonmonotone}]

\

Case 1. For any $s>s^*$ be fixed, and $\lambda_1$ is defined. We define a $c$-independence value $q_1$ as follows
    $$q_1=\frac{2}{-(\mu_1 
\lambda_1)^2+s(\mu_1 \lambda_1)-1} > \frac{1+ac}{-(\mu_1\lambda_1)^2+s(\mu_1 \lambda_1)-1}$$
where $\mu_1$ is to be determined later. This inequality holds because of the weak competition condition $b<a<\frac{1}{c}$. Note that this $q_1$ may not satisfies (A2). So, we pick $\mu_1$ close to $1$ such that   
$$q_1=\frac{2}{-(\mu_1 
\lambda_1)^2+s(\mu_1 \lambda_1)-1} >1$$
that is we have
$$q_1 > \max\{1,\frac{1+ac}{-(\mu_1\lambda_1)^2+s(\mu_1 \lambda_1)-1}\}$$
which is independent on $c$. Finally, we choose a suitable $\delta_1$ that satisfy (A3).

Then by $(\ref{maxoff})$ we obtain that the global maximum of $\underline{u}(\xi)$ is
$$||\underline{u}||_\infty=(1-\frac{1}{\mu_1})(q_1 \mu_1)^{\frac{-1}{\mu_1-1}}$$
    which is a constant depending on $s$  and is independent on $c$. Therefore, there exists a $\delta(a,b,s)>0$ such that if $\delta(a,b,s)<c<\frac{1}{a}$ then
$$u^*:=\frac{1-ac}{1-bc}<(1-\frac{1}{\mu_1})(q_1 \mu_1)^{\frac{-1}{\mu_1-1}}=||\underline{u}||_\infty.$$
By the asymptotic behavior of $(u,v)$, Lemma \ref{shrinking}, we can see that the profile of $u(\xi)$ is a non-monotone function.

\

Case 2. Let $s=s^*$ and $\hat{\lambda}_1$ is defined. We note that, in both cases ($ad<1$ and $ad=1$), the lower solution $\underline{u}(\xi)$ is defined in the same form. 
$$\underline{u}(\xi)=[-h\xi-q\sqrt{-\xi}]e^{\lambda \xi}.$$
Since $c<\frac{1}{a}$ we can always choose
$$\hat{q}_1=\max \left\{ \sqrt{\hat{h}_1(\frac{1}{\hat{\lambda}_1}+1)}, \ 4\left( \frac{1}{a}  \hat{h}_1 \hat{h}_2 (\frac{7}{2e\hat{\lambda}_2})^{\frac{7}{2}}  + \hat{h}_1^2(\frac{7}{2e\hat{\lambda}_1})^{\frac{7}{2}} \right) \right\}$$
which satisfies (B2) and is independent on $c$.

Moreover, the maximum of $g(\xi)$, 
\begin{equation}
    g(\xi)=(-h \xi-q \sqrt{-\xi})e^{\lambda \xi}, \ \ \xi \leq -(\frac{q}{h})^2\notag
\end{equation}
defined in $(\ref{generalsupsubmin})$, depends only on $\lambda, h$ and $q$. In other words, $||\underline{u}||_\infty$ is also independent of $c$. Applying the same technique, we also have the same conclusion.
\end{proof}
The proof of Theorem \ref{vnonmonotone} follows by a similar argument. The key point is that under the weak competition condition $b<a$, we can choose $q_2$, $\hat{q}_2$ and $\hat{Q}_2$
 sufficiently large so that they satisfy (A2), (B3), and (B6), respectively; in addition, $q_2$, $\hat{q}_2$ and $\hat{Q}_2$ can be chosen so that they are independent of $b$. Therefore, $||\underline{v}||_\infty$ is independent of $b$. We omit the proof here.

The following is an example demonstrating a traveling wave solution that is not monotonically increasing.

\begin{example}
    Set $d=1, a=1,  \ c=\frac{1}{2},$ and given any $s>2$. Then set $b_n=1-\frac{1}{n}$ for $n > 2 (k+1)(q_2(k)\frac{k+1}{k})^k-1$ where
    $$k>\max\{1,(\min\{ \frac{\lambda_4}{\lambda_2}, \frac{\lambda_1+\lambda_2}{\lambda_2},2\}-1)^{-1}\}$$
    and
    $$q_2(k)=\frac{2}{-((1+\frac{1}{k})\lambda_2)^2+s((1+\frac{1}{k})\lambda_2)-1}.$$
    If $(u,v)$ is a solution in Lemma $\ref{existofsupsub}$ with such $s$, then the profile of $v(\xi)$ is a non-monotone wave.
\end{example}

\begin{proof}

 Setting $a=1, \ b_n=1-\frac{1}{n}, \ c=\frac{1}{2}, \ d=1$ so that the parameter satisfies the weak competition condition $(\ref{weak})$ for some very large $n \in \mathbb{R}^+$. The equilibrium $(u^*,v^*)$ is equal to
$$(u^*,v^*)=(\frac{n}{n+1},\frac{2}{n+1}).$$

For $s>s^*=2$, the values $\lambda_1, \lambda_2$ are defined as
$$\lambda_1=\lambda_2=\frac{s-\sqrt{s^2-4}}{2}>0.$$

\

Set $\mu_2=1+\frac{1}{k} < \min\{ \frac{\lambda_4}{\lambda_2}, \frac{\lambda_1+\lambda_2}{\lambda_2},2\}$ for some $k>1$ and $q_2=q_2(k) = \frac{2}{-((1+\frac{1}{k})\lambda_2)^2+s((1+\frac{1}{k})\lambda_2)-1}$. Then we obtain that
$$||\underline{v}||_\infty=(\frac{1}{k+1})(q_2(k)\frac{k+1}{k})^{-k}.$$
Therefore, if 
$$v^*=\frac{2}{n+1}<(\frac{1}{k+1})(q_2(k)\frac{k+1}{k})^{-k}$$
then the profile of $v(\xi)$ is a non-monotone wave.

\end{proof}
\begin{remark}
    Let $s=4.5$. Then for $\lambda_1=\lambda_2=\frac{s-\sqrt{s^2-4}}{2}\approx 0.234$, the pair of functions

\begin{equation*}
    \Bar{v}(\xi)=
     \begin{cases}
         1 \ & \text{ if } \xi \geq 0,\\
    e^{\lambda_2 \xi} \ & \text{ if } \xi <0,
    \end{cases}
\end{equation*}
and
\begin{align*}
    \underline{v}(\xi)=
     \begin{cases}
         \delta_2 \ & \text{ if } \xi \geq -5,\\
    e^{\lambda_2 \xi}-q_2e^{\mu_2 \lambda_2 \xi} \ & \text{ if } \xi \leq -5,
     \end{cases}
\end{align*}
are the sup-sub solutions of $(\ref{supsub})$ where we pick $q_2=2.6, \ \mu_2=1+\frac{1}{1.1} \approx 1.91$. By direct calculation, we have $\displaystyle\sup_{\xi \in \mathbb{R}}\underline{v}(\xi)\approx 0.0817 >v^*=\frac{2}{n+1}$ if $n \geq 24$. That is, for any $n \geq 24$, there exists a non-monotone wave. For Figure \ref{Fig 1} is the profile of $\Bar{v}(\xi), \underline{v}(\xi)$ and $v^*$.
\end{remark}

\begin{figure}[h]
\centering
\includegraphics[width=9cm]{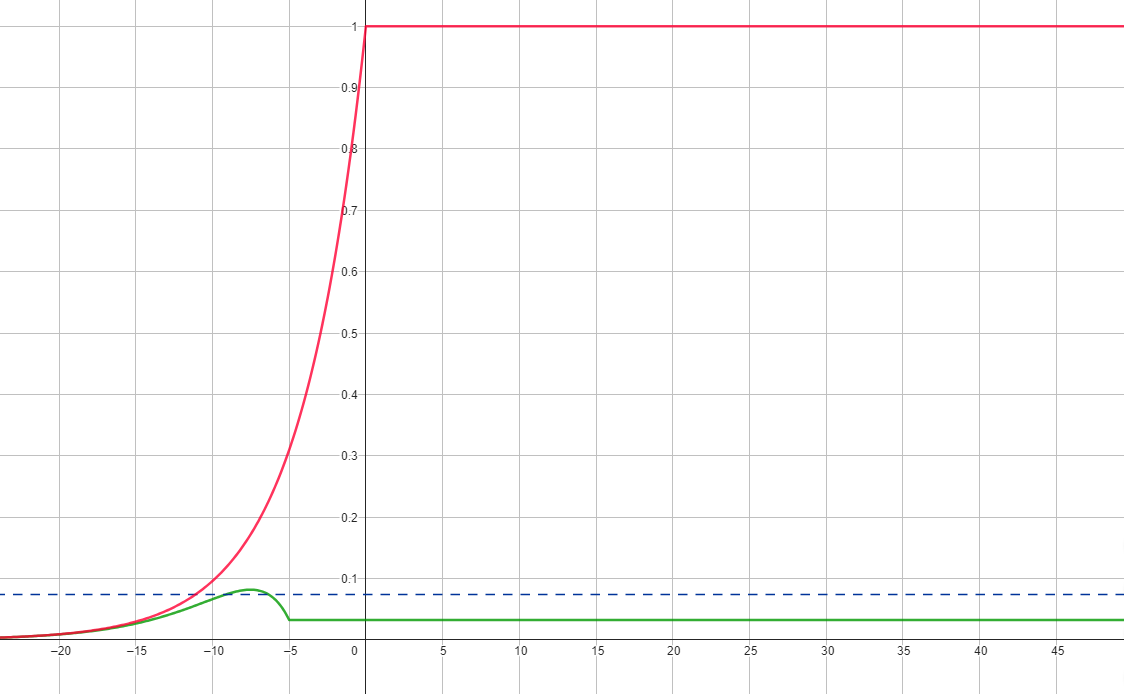}
\caption{When $n=26$, $\Bar{v}(\xi)$(red line),$\underline{v}(\xi)$(green line),$v^*=\frac{2}{27}$(blue dash line) are all labeled on the figure. There exists $v(\xi)$ lying between the red line and green line with $\displaystyle\lim_{\xi \to +\infty}v(\xi)=v^*$.}
\label{Fig 1}
\end{figure}

\begin{remark}
    Note that given any $\underline{s}>s^*$, for any $s \geq \underline{s}$ we can choose $\mu_2$ very close to $1$ such that $(\ref{maineq})$ and $(\ref{asymptotic behavior})$ has a non-monotone wave whenever $b$ is close to $a$. Figure \ref{Fig 2}. is the illustration.
\end{remark}

\begin{figure}[h]
\centering
\includegraphics[width=10cm]{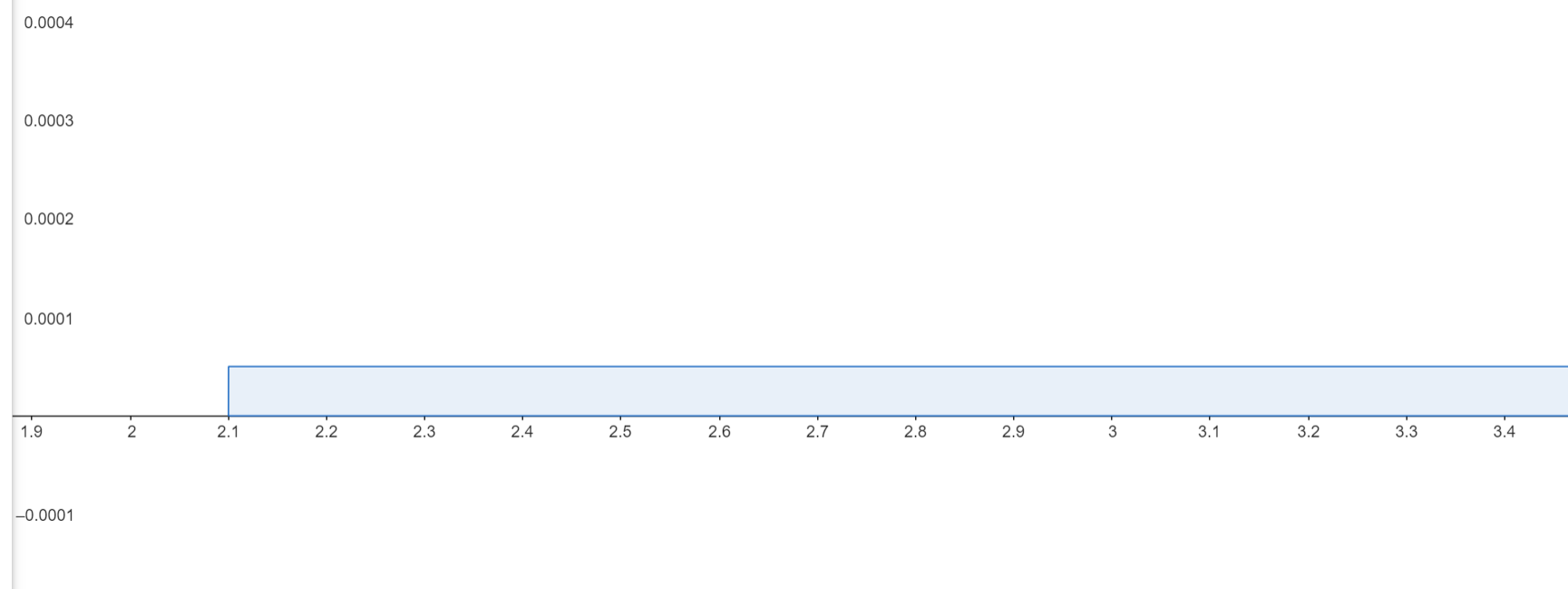}
\caption{Set $d=1, a=1, c=\frac{1}{2}, \underline{s}=2.1, \mu_2=1.001$. The horizontal axis is the wave speed $s$, and the vertical axis represents the difference $a-b>0$. The blue area illustrates the region where non-monotone solutions $v(\xi)$ exist.}
\label{Fig 2}
\end{figure}

\newpage

Moreover, on the non-monotone solution, we can prove the following basic property.

\begin{proposition}
    If $u(\xi)$ (resp., $v(\xi)$) is oscillating near $ +\infty$, then so is $v(\xi)$ (resp., $u(\xi)$).
\end{proposition}

\begin{proof}
    Assume that $u(\xi)$ is oscillating near $ +\infty$ and $v(\xi)$ is eventually monotone. If $v(\xi)$ is eventually monotone increasing at $+\infty$, then by assumption, there exists a sequence $I_k=[\xi_k^m,\xi_k^M]$ such that $u(\xi)$ is strictly monotonic increasing on $I_k$ with $u(\xi_k^m)$ is a local minimum, $u(\xi_k^M)$ a local maximum and $0<u(\xi_k^m) < u(\xi_k^M)$. By $(\ref{maineq})$, we have
\[
\begin{cases}
    u(\xi_k^M)(1-u(\xi_k^M)-cv(\xi_k^M)) \geq 0,\\
    u(\xi_k^m)(1-u(\xi_k^m)-cv(\xi_k^m)) \leq 0,
\end{cases}
\]
which is equivalent to
\begin{equation}\label{osci 1}
\begin{cases}
    v(\xi_k^M) \leq \frac{1-u(\xi_k^M)}{c},\\
    v(\xi_k^m) \geq \frac{1-u(\xi_k^m)}{c}.
\end{cases}
\end{equation}
Since $v$ is monotone increasing, $(\ref{osci 1})$ implies that $u(\xi_k^M) \leq u(\xi_k^m)$, a contradiction. Similarly, we can also show that $v(\xi)$ is not eventually monotone decreasing at $+\infty$.
\end{proof}

At present, for non-monotone wave, we can only characterize the above property. To determine the precise increasing–decreasing behavior, more refined tools are required, which is one of our future research directions. In fact, we can use this type of non-monotone wave to construct the front-pulse solution.

\section{Front-pulse waves at critical strong-weak} competition \label{Section 4}
We now consider two degenerate cases of our model under specific parameter settings.

First, when $b<a$ and $c=\frac{1}{a}$, the original system reduces to the following degenerate model:

\begin{equation} \label{degenerate1}
\begin{cases}
    u''-su'+u(1-u-\frac{1}{a}v)=0, \ \xi \in \mathbb{R},\\
    dv''-sv'+v(a-bu-v)=0,\\
    (u,v)(-\infty)=(0,0), \ (u,v)(+\infty)=(0,a).
\end{cases}
\end{equation}

It is straightforward to verify that for any $s\geq 2\sqrt{ad}$, the system $(\ref{degenerate1})$ admits a trivial traveling wave solution $(u(\xi),v(\xi))$, where $u(\xi)=0$ for all $\xi \in \mathbb{R}$ and $v(\xi)$ satisfies the classical Fisher-KPP equation:
\begin{equation} \label{vFKPP}
\begin{cases}
    dv''-sv'+v(a-v)=0,\\
    v(-\infty)=0, \ v(+\infty)=a.
\end{cases}
\end{equation}

On the other hand, when $b=a$ and $a<\frac{1}{c}$, the original system reduces to another degenerate model:

\begin{equation} \label{degenerate2}
\begin{cases}
    u''-su'+u(1-u-cv)=0, \ \xi \in \mathbb{R},\\
    dv''-sv'+v(a(1-u)-v)=0,\\
    (u,v)(-\infty)=(0,0), \ (u,v)(+\infty)=(1,0).
\end{cases}
\end{equation}
Similarly, for any $s\geq 2$, $(\ref{degenerate2})$ admits a trivial traveling wave solution $(u(\xi),v(\xi))$, where $v(\xi)=0$ for all $\xi \in \mathbb{R}$ and $u(\xi)$ satisfies the classical Fisher-KPP equation:
\begin{equation} \label{vFKPP}
\begin{cases}
    u''-su'+u(1-u)=0,\\
    u(-\infty)=0, \ u(+\infty)=1.
\end{cases}
\end{equation}

We call a non-trivial solution of $(\ref{degenerate1})$ or $(\ref{degenerate2})$ front-pulse solution. This type of front-pulse solution is very rare and interesting in biological models. We can obtain such type of solutions using Theorem $\ref{unonmonotone}$ and Theorem $\ref{vnonmonotone}$. To proceed with this construction, we first introduce the following lemmas.

\begin{lemma} 
    Given any $\Omega=[r,t] \subset \mathbb{R}$. Under the same assumption of Theorem $\ref{unonmonotone}$, the non-monotone solution $(u(\xi),v(\xi))$ of $(\ref{maineq})$ and $(\ref{asymptotic behavior})$ satisfies the following interior estimate
    $$||u||_{C^{3,\alpha}(\Omega)}, \ ||v||_{C^{3,\alpha}(\Omega)} \leq C(\Omega,\alpha,s,d,a,b),$$
    where $\alpha \in (0,1)$ and the constant is independent of $c$.
\end{lemma}

\begin{proof} 
Fix $s \geq \max\{2,2\sqrt{ad}\}$. We rewrite the $u$ equation as $Lu=f(\xi)$, where $Lu=u''-su'+u$ is a linear operator and $f(\xi):=u^2(\xi)+cu(\xi)v(\xi)$. 

First, by the maximum principle, we have $0<u(\xi)<1$ and $0<v(\xi)<a$ for all $\xi \in \mathbb{R}$. On the other hand, the source term $f(\xi)$ satisfies $|f(\xi)| \leq |u^2(\xi)|+u^2(\xi)+c|u(\xi)v(\xi)| \leq 2$ for all $\xi \in \mathbb{R}$. 

Given any $\Omega=[r,t]$. We pick a larger set $\hat{\Omega}=[r-1,t+1]$ and apply the $L^p$ theory. For any $1<p<\infty$, there exists $C_p(\Omega,s)$ such that
$$||u||_{W^{2,p}(\Omega)} \leq C_p(||f||_{L^p(\hat{\Omega})}+||u||_{L^p(\hat{\Omega})}).$$
Since $\hat{\Omega}$ is compact, $||u||_{L^p(\hat{\Omega})} \leq |\hat{\Omega}|^{\frac{1}{p}}$ and $||f||_{L^p(\hat{\Omega})} \leq 2|\hat{\Omega}|^{\frac{1}{p}}$. This implies $||u||_{W^{2,p}(\Omega)} \leq 3C_p|\hat{\Omega}|^{\frac{1}{p}}$. That is $||u||_{W^{2,p}(\Omega)}$ is bounded and the upper bound is independent on $c$.

Next, by the Sobolev embedding theorem, $W^{2,p}(\Omega) \hookrightarrow C^{1,\alpha}(\Omega)$ for $p>1$ and $\alpha=1-\frac{1}{p}$. We have
$$||u||_{C^{1,\alpha}(\Omega)} \leq C||u||_{W^{2,p}(\Omega)} \leq C_0$$
for some $C_0$ is independent on $c$. 

Finally, by the equation of $u$, we also have
$$||u''||_{C^0(\Omega)} \leq ||f||_{C^0(\Omega)}+|s|\cdot ||u'||_{C^0(\Omega)} \leq C_1$$
for some $C_1$ is independent on $c$. We finally obtain that $u \in C^2(\Omega)$. Similarly, consider $v$ equation, we also have $v \in C^2(\Omega)$. And both $C^2$ norm upper bound are independent on $c$.

Now, 
\begin{align*}
    ||f'(\xi)||_{C^0(\Omega)}&=||cu'(\xi)v(\xi)+cu(\xi)v'(\xi)+2u(\xi)u'(\xi)||_{C^0(\Omega)}\notag\\
    &\leq \frac{1}{a}||u'(\xi)v(\xi)+u(\xi)v'(\xi)||_{C^0(\Omega)}+2||u(\xi)u'(\xi)||_{C^0(\Omega)}\notag\\
    &\leq C_0(\Omega,\alpha,s,d,a,b).\notag
\end{align*}
is a finite number. Hence, $f \in C^1(\Omega)$ and the bound is independent on $c$. 

For $f(\xi) \in C^1(\Omega)$, applying the same argument as above, we have  $u$ and $v$ belong to  $C^{3,\alpha}(\Omega)$ with the bound is independent on $c$.

\end{proof}

    

Using a similar argument, we can establish the following result.

\begin{lemma} 
    Given any compact set $\Omega \subset \mathbb{R}$. Under the same assumption of Theorem $\ref{vnonmonotone}$, the non-monotone solution $(u(\xi),v(\xi))$ of $(\ref{maineq})$ and $(\ref{asymptotic behavior})$ satisfies the following interior estimate
    $$||u||_{C^{3,\alpha}(\Omega)}, \ ||v||_{C^{3,\alpha}(\Omega)} \leq C(\Omega,\alpha,s,d,a,c),$$
    where $\alpha \in (0,1)$ and the constant is independent of $b$.
\end{lemma}

We now turn to the proof of Theorem \ref{nontrivialu}.

\begin{proof}[Proof of Theorem \ref{nontrivialu}]
    Fix any $b<a, \ d>0, \ s \geq s^*$. By Theorem $\ref{unonmonotone}$, there exists a pair of positive sup-sub solution $(\overline{u},\underline{u}),(\overline{v},\underline{v})$ and $\delta>0$ such that if $\delta=\delta(a,b,s)<c<\frac{1}{a}$, then there exists a solution $(u(\xi),v(\xi))$ of $(\ref{maineq})$ and $(\ref{asymptotic behavior})$, where $u(\xi)$ is a non-monotone function. Note that we choose 
    $$||\underline{u}||_{\infty}=(1-\frac{1}{\mu_1})(q_1 \mu_1)^{\frac{-1}{\mu_1-1}}$$
    which is independent of $c$. For any $c_k \in (\delta,\frac{1}{a})$ with $c_k \to \frac{1}{a}$, there exists a non-monotone solution $(u_k(\xi),v_k(\xi))$ with
    $$||u_k||_{\infty}>||\underline{u}||_{\infty}>0.$$
    This implies that $u_k(\xi)$ has a positive lower bounded function $\underline{u}(\xi)$ with $||\underline{u}||_{\infty}$ which is independent of $c$. By the boundedness of the $C^{3,\alpha}$ norm of $(u_k,v_k)$, there exists a subsequence, $(u_{k_j}(\xi),v_{k_j}(\xi),c_{k_j})$ converges to  $(U(\xi),V(\xi),\frac{1}{a})$ in $C^2 \times C^2 \times \mathbb{R}$ on any compact set $[-M,M]$ as $j \to +\infty$, for some $(U,V)$ satisfies the limiting equation
    \begin{equation}
    \begin{cases}
        U''-sU'+U(1-U-\frac{1}{a}V)=0, \  \xi \in [-M,M]\\
        dV''-sV'+V(a-bU-V)=0,
    \end{cases}
    \end{equation}
    with $||U||_{\infty} \geq ||\underline{u}||_{\infty}>0$, and $||U||_{C^2(\Omega)} \leq C(\Omega,\alpha,s,d,a,b)$. Note that $(u_{k_j},v_{k_j})$ satisfies $(\ref{asymptotic behavior})$. When $M \to +\infty$, we can see that $(U,V)$ satisfies 
    $$\lim_{\xi \to -\infty}(U(\xi),V(\xi))=(0,0),$$
    and \ $ 0 \leq U(\xi) \leq 1, \  \delta_2 \leq V(\xi) \leq a$ for $\xi$ very large, where $U(\xi)$ is a non-trivial function. Now we claim that $\displaystyle\lim_{\xi \to +\infty} (U(\xi),V(\xi))=(0,a)$.

    Case 1. If $U(\xi)$ and $V(\xi)$ are eventually monotone, then assume $\displaystyle\lim_{\xi \to +\infty} (U(\xi),V(\xi))=(\alpha,\beta)$. It is easy to show that the derivatives vanish. Therefore, $\alpha, \beta$ satisfy 
\[\begin{cases}
    \alpha(1-\alpha-\frac{1}{a}\beta)=0,\\
    \beta(a-b\alpha-\beta)=0,
\end{cases}
\]
with $0 \leq \alpha \leq 1,$ and $\delta_2 \leq \beta \leq a$. This implies that $(\alpha,\beta)=(0,a)$.

\

Case 2. If $V(\xi)$ oscillates and $U(\xi)$ is eventually monotonic at $+\infty$, and let $\displaystyle\lim_{\xi \to +\infty}U(\xi)=\alpha$. By assumption, there exist a sequence $I_k=[\xi_k^m,\xi_k^M]$ such that $V(\xi)$ is monotone increasing on $I_k$, $V(\xi_k^m)$ is a local minimum, $V(\xi_k^M)$ is a local maximum and $V(\xi_k^m) \leq V(\xi_k^M)$. By $(\ref{maineq})$, we have
\[
\begin{cases}
    V(\xi_k^M)(a-bU(\xi_k^M)-V(\xi_k^M)) \geq 0,\\
    V(\xi_k^m)(a-bU(\xi_k^m)-V(\xi_k^m)) \leq 0,
\end{cases}
\]
which is equivalent to
\begin{equation}\label{degenerate Case 2}
\begin{cases}
    U(\xi_k^M) \leq \frac{a-V(\xi_k^M)}{b},\\
    U(\xi_k^m) \geq \frac{a-V(\xi_k^m)}{b}.
\end{cases}
\end{equation}

Note that $(\ref{degenerate Case 2})$ implies that $U(\xi)$ is monotone decreasing. If not, we have
$$\frac{a-V(\xi_k^m)}{b} \leq U(\xi_k^m) \leq U(\xi_k^M) \leq \frac{a-V(\xi_k^M)}{b},$$
and we have $V(\xi_k^M) \leq V(\xi_k^m)$, a contradiction. Since $\displaystyle\lim_{\xi \to +\infty}U(\xi)=\alpha$, $(\ref{degenerate Case 2})$ shows that $\displaystyle\lim_{\xi \to +\infty}V(\xi)=a-b\alpha:=\beta$ exist. If $\alpha>0$, then
$$\lim_{\xi\to+\infty}U(\xi)(1-U(\xi)-\frac{1}{a}V(\xi))=\frac{(b-a)\alpha^2}{a}<0.$$
On the other hand, since $U(\xi)$ is monotone decreasing, we can show that $\displaystyle\lim_{\xi \to +\infty}U''(\xi)=0$ and $\displaystyle\lim_{\xi \to +\infty}U'(\xi)=0$. Therefore,
$$0=\lim_{\xi \to +\infty}(U''(\xi)-sU'(\xi)+U(\xi)(1-U(\xi)-\frac{1}{a}V(\xi)))<0,$$
a contradiction. Thus, $\alpha=0$ and $\beta=a$.

\

Case 3. If $U(\xi)$ oscillates and $V(\xi)$ is eventually monotonic at $+\infty$, and let $\displaystyle\lim_{\xi \to +\infty}V(\xi)=\beta$.

Case 3.1. When $\beta=a$, then we can apply similar argument to show that $\displaystyle\lim_{\xi \to +\infty}U(\xi)=0$.

\

Case 3.2. If $\beta<a$, then either $V$ is monotone increasing or decreasing, we can show that $\displaystyle\lim_{\xi \to +\infty}V''(\xi)=\displaystyle\lim_{\xi \to +\infty}V'(\xi)=0$. Thus, by $(\ref{maineq})$, we have
\begin{equation}
    0=\lim_{\xi \to +\infty}V(\xi)(a-bU(\xi)-V(\xi)). \notag 
\end{equation}
This implies
\begin{equation}
    \lim_{\xi \to +\infty}U(\xi)=\lim_{\xi \to +\infty}(V(\xi)U(\xi))(\frac{1}{V(\xi)})=\frac{a-\beta}{b}>0.\notag
\end{equation}
Since $\displaystyle\lim_{\xi \to +\infty}U(\xi)$ exists, and it deduces to Case 1.

\

Case 4. If both $U(\xi)$ and $V(\xi)$ oscillate at $+\infty$. 

Case 4.1. If $\displaystyle\lim_{\xi \to +\infty}V(\xi)=a$. For any local maximum points $\xi_k^M$, of $U$. By $(\ref{maineq})$, we have
\begin{equation}
    U(\xi_k^M)(1-U(\xi_k^M)-\frac{1}{a}V(\xi_k^M)) \geq 0,\notag
\end{equation}
which is equivalent to
\begin{equation}
     0 \geq -U^2(\xi_k^M) \geq U(\xi_k^M)(\frac{1}{a}V(\xi_k^M)-1).\notag
\end{equation}
By the Squeeze Theorem, we have $\displaystyle\limsup_{\xi \to +\infty}U(\xi)=0$. This implies that $\displaystyle\lim_{\xi \to +\infty}U(\xi)=0$.

Case 4.2. If $\displaystyle\liminf_{\xi \to +\infty}V(\xi)=\beta<a$, and $\displaystyle\limsup_{\xi \to +\infty}U(\xi)=\alpha>0$. For any small $\frac{1}{n}>0$ there exist $\xi_n$ which be the local minimum points of $V$ such that $V(\xi_n) \leq \beta+\frac{1}{n}$. By $(\ref{maineq})$, we have
\begin{equation}
   0 \geq a-bU(\xi_n)-V(\xi_n) \geq a-bU(\xi_n)-\beta-\frac{1}{n},\notag
\end{equation}
which implies that
\begin{equation}
      U(\xi_n) \geq \frac{a-\beta-\frac{1}{n}}{b}.\notag
\end{equation}
Therefore, $\displaystyle\limsup_{\xi \to +\infty}U(\xi)=\alpha \geq \frac{a-\beta}{b}$. On the other hand, up to subsequence, for any small $\frac{1}{k}>0$, there exist the local maximum points, $\xi_k$,  of $U$ such that $0<\alpha-\frac{1}{k}<U(\xi_k)<\alpha+\frac{1}{k}$. Again, by $(\ref{maineq})$,  we have
\begin{equation}
   0 \leq 1-U(\xi_k)-\frac{1}{a}V(\xi_k)<1-\alpha+\frac{1}{k}-\frac{1}{a}V(\xi_k) \leq 1-\frac{a-\beta}{b}+\frac{1}{k}-\frac{\beta}{a}.\notag
\end{equation}
Let $k \to +\infty$, we then obtain that
\begin{equation}
    0 \leq 1-\frac{a}{b}+\frac{\beta}{b}-\frac{\beta}{a},\notag
\end{equation}
which is equivalent to
\begin{equation}
    a(\frac{1}{b}-\frac{1}{a}) \leq \beta (\frac{1}{b}-\frac{1}{a}), \notag
\end{equation}
a contradiction. Hence, $\displaystyle\lim_{\xi \to +\infty}V(\xi)=a$ and $\displaystyle\lim_{\xi \to +\infty}U(\xi)=0$.

\end{proof}

\begin{remark}
    As $c_{k_j} \to \frac{1}{a}$, the definition of $\underline{u}(\xi)$ changes. Due to the inequality $0<\delta_{1_{k_j}}<u^*_{k_j}$, we in fact have $\displaystyle\lim_{j \to +\infty}\delta_{k_j}=0$. However, this does not affect the maximum value $||\underline{u}||_{\infty}$. 
\end{remark}

\begin{proposition} \label{v pulse}
    For any $a<\frac{1}{c}, \ a=b, \ d>0$. If $s \geq s^*$, there exists a non-trivial front-pulse solution of $(\ref{degenerate2})$ with $v(\xi) \to 0$ as $|\xi|\to +\infty$.
\end{proposition}

The proof is very similar to the proof of Theorem $\ref{nontrivialu}$. We omit the details here.

\section{Acknowledgments}
Chiun-Chuan Chen
is supported by the National Science and Technology Council, Taiwan (Grant Number
114-2115-M-002 -004 -MY3) and the National Center for Theoretical Sciences, Taiwan
(NCTS). Ting-Yang Hsiao is supported by the ERC CONSOLIDATOR GRANT 2023 ``Generating Unstable Dynamics in dispersive Hamiltonian fluids'' , Project Number: 101124921.
 Views and opinions expressed are, however, those of the authors only and do not necessarily reflect those of the European Union or the European Research Council. Neither the European Union nor the granting authority can be held responsible for them.

\appendix

\section{Proof of nonexistence for $s\in(0,s^*)$} \label{Appendix A}

For readers' convenience, we provide a direct proof of the nonexistence of traveling waves when $0<s<s^*$.

\begin{theorem}
   When $0< s < s^*$, there is no positive solution $(u,v)$ of the equation $(\ref{maineq})$ and $(\ref{asymptotic behavior})$.
\end{theorem}

\begin{proof}
     Without loss of generality, we assume $s^*=2$. For contradiction, if for some $0<s<2$ there exists a positive solution $(u,v)(\xi)$ of $(\ref{maineq})$ and $(\ref{asymptotic behavior})$. Now , we define the positive function $w(\xi)=e^{\frac{-s \xi}{2}}u(\xi)$. By the product rule,
     \[
     w''(\xi)=e^{\frac{-s\xi}{2}}[\frac{s^2}{4}u-su'(\xi)+u''(\xi)]
     \]
     Then by the $u$-equation of $(\ref{maineq})$, $w(\xi)$ satisfies
    \begin{equation}{\label{wequation}}
        w''(\xi)+w(\xi)(1-\frac{s^2}{4}-u(\xi)-cv(\xi))=0.
    \end{equation}
    Since $0<s<2$ and $(u,v)$ tends to zero as $\xi \to -\infty$, there exist small $\epsilon>0$ and $-L<0$ such that
    $$(1-\frac{s^2}{4}-u(\xi)-cv(\xi))>\epsilon, \ \text{ for all }\xi<-L.$$
    We define an auxiliary function $\phi(\xi)=\sin(\sqrt{\epsilon}\xi)$ which is a positive solution of the following linear boundary value problem
    \begin{equation}\label{phiequation}
    \begin{cases}
        \phi''+\epsilon\phi=0,\\
        \phi(\frac{-2M\pi}{\sqrt{\epsilon}})=\phi(\frac{-(2M-1)\pi}{\sqrt{\epsilon}})=0,
    \end{cases}
    \end{equation}
    where $M\in \mathbb{N}$ is a large number such that $\frac{-(2M-1)\pi}{\sqrt{\epsilon}}<-L$. Set $\frac{-2M\pi}{\sqrt{\epsilon}}=\xi_1, \ \frac{-(2M-1)\pi}{\sqrt{\epsilon}}=\xi_2$. Multiply equation $(\ref{wequation})$ by $\phi(\xi)$ and multiply equation $(\ref{phiequation})$ by $w(\xi)$, then subtract equation $(\ref{phiequation})$ from equation $(\ref{wequation})$. We have
    \begin{equation}
        w''(\xi)\phi(\xi)-\phi''(\xi)w(\xi)+w(\xi)\phi(\xi)(1-\frac{s^2}{4}-u(\xi)-cv(\xi)-\epsilon)=0\notag
    \end{equation}
    Integrating both sides of the equation from $\xi_1$ to $\xi_2$, we then obtain that
    \begin{equation}\notag
        \left[w'(\xi)\phi(\xi)-\phi'(\xi)w(\xi)\right]\Big|_{\xi_1}^{\xi_2}+\int_{\xi_1}^{\xi_2}w(\xi)\phi(\xi)(1-\frac{s^2}{4}-u(\xi)-cv(\xi)-\epsilon)d\xi=0 .
    \end{equation}
    The first term is positive, and the second term is also positive, a contradiction.
\end{proof}

\begin{center}
\bibliographystyle{alpha}
\bibliography{Bibjournal.bib}

@article{MP1980,
  title={Propagating fronts for competing species equations with diffusion},
  author={Tang, Min Ming and Fife, Paul C},
  journal={Archive for Rational Mechanics and Analysis},
  volume={73},
  number={1},
  pages={69--77},
  year={1980},
  publisher={Citeseer}
}

@article{lin2014traveling,
  title={Traveling wave solutions for delayed reaction-diffusion systems and applications to diffusive {L}otka-{V}olterra competition models with distributed delays},
  author={Lin, Guo and Ruan, Shigui},
  journal={Journal of Dynamics and Differential Equations},
  volume={26},
  pages={583--605},
  year={2014},
  publisher={Springer}
}

@article{Kanon,
  title={Parameter dependence of propagation speed of travelling waves for competition-diffusion equations},
  author={Kan-On, Yukio},
  journal={SIAM Journal on Mathematical Analysis},
  volume={26},
  number={2},
  pages={340--363},
  year={1995},
  publisher={SIAM}
}

@article{JY,
  title={The sign of the wave speed for the {L}otka-{V}olterra competition-diffusion system},
  author={Guo, Jong-Shenq and Lin, Ying-Chih},
  journal={Commun. Pure Appl. Anal},
  volume={12},
  number={5},
  pages={2083--2090},
  year={2013}
}

@article{S.Ma,
  title={Traveling wavefronts for delayed reaction-diffusion systems via a fixed point theorem},
  author={Ma, Shiwang},
  journal={Journal of Differential Equations},
  volume={171},
  number={2},
  pages={294--314},
  year={2001},
  publisher={Elsevier}
}

@article{YLHGLin,
  title={Traveling wave solutions in a diffusive system with two preys and one predator},
  author={Huang, Yan-Li and Lin, Guo},
  journal={Journal of Mathematical Analysis and Applications},
  volume={418},
  number={1},
  pages={163--184},
  year={2014},
  publisher={Elsevier}
}

@article{CTS,
  title={Savanna dynamics with grazing, browsing, and migration effects},
  author={Chen, Chiun-Chuan and Hsiao, Ting-Yang and Wang, Shun-Chieh},
  journal={SIAM Journal on Applied Dynamical Systems},
  volume={24},
  number={4},
  pages={2950--2976},
  year={2025},
  publisher={SIAM}
}

@article{WGS2006,
  title={Existence of travelling wave solutions in delayed reaction-diffusion systems with applications to diffusion--competition systems},
  author={Li, Wan-Tong and Lin, Guo and Ruan, Shigui},
  journal={Nonlinearity},
  volume={19},
  number={6},
  pages={1253},
  year={2006},
  publisher={IOP Publishing}
}

@article{gg2022,
  title={Traveling waves for a three-species competition system with two weak aboriginal competitors},
  author={Guo, Jong-Shenq and Guo, Karen},
  journal={Differential and Integral Equations},
  volume={35},
  number={11/12},
  pages={819--832},
  year={2022},
  publisher={Khayyam Publishing, Inc.}
}

@article{sonego2024control,
  title={Control of a {L}otka-{V}olterra System with Weak Competition},
  author={Sonego, Maicon and Zuazua, Enrique},
  journal={arXiv preprint arXiv:2409.20279},
  year={2024}
}

@article{chang2025bistable,
  title={Bistable wavefronts of three-species competition-diffusion systems: {W}eak competition strategy between two species},
  author={Chang, Chueh-Hsin and Wu, Chang-Hong},
  journal={Journal of Differential Equations},
  volume={422},
  pages={529--561},
  year={2025},
  publisher={Elsevier}
}

@article{chang2022series,
  title={A Series Evans Function Approach to Stability of Traveling Waves of Reaction-Diffusion Systems},
  author={Chang, Chueh-Hsin and Hsu, Cheng-Hsiung and Yang, Tzi-Sheng},
  journal={Available at SSRN 4339817},
  year={2022}
}

@article{Yang,
  title={Non-monotone travelling wave solutions for the two-species {L}otka-{V}olterra competitive system with diffusion},
  author={Yang, Che-Wei},
  journal={Master's Thesis},
  pages={1-19},
  year={2022},
  publisher={National Taiwan University}
}

@article{chang2023propagating,
  title={Propagating direction in the two species {L}otka-{V}olterra competition-diffusion system},
  author={Chang, Mao-Sheng and Chen, Chiun-Chuan and Wang, Shun-Chieh},
  journal={Discrete and Continuous Dynamical Systems-B},
  volume={28},
  number={12},
  pages={5998--6014},
  year={2023},
  publisher={Discrete and Continuous Dynamical Systems-B}
}

@article{peng2021sharp,
  title={Sharp estimates for the spreading speeds of the {L}otka-{V}olterra diffusion system with strong competition},
  author={Peng, Rui and Wu, Chang-Hong and Zhou, Maolin},
  booktitle={Annales de l'Institut Henri Poincar{\'e} C, Analyse non lin{\'e}aire},
  volume={38},
  number={3},
  pages={507--547},
  year={2021},
  organization={Elsevier}
}

@article{jiang2025spreading,
  title={Spreading Speeds for A Diffusive Three-Species System with Strong Competition},
  author={Jiang, Jyun-Sheng and Wu, Chang-Hong},
  journal={Journal of Dynamics and Differential Equations},
  pages={1--30},
  year={2025},
  publisher={Springer}
}

@article{morita2009entire,
  title={An entire solution to the {L}otka-{V}olterra competition-diffusion equations},
  author={Morita, Yoshihisa and Tachibana, Koichi},
  journal={SIAM Journal on Mathematical Analysis},
  volume={40},
  number={6},
  pages={2217--2240},
  year={2009},
  publisher={SIAM}
}

@article{hsiao2022estimates,
  title={Estimates of Population Size for Traveling Wave Solutions of Spatially Non-local {L}otka-{V}olterra Competition System},
  author={Hsiao, Ting-Yang},
  journal={Journal of Dynamics and Differential Equations},
  volume={34},
  number={3},
  pages={1969--1996},
  year={2022},
  publisher={Springer}
}

@article{hung2016n,
  title={N-barrier maximum principle for degenerate elliptic systems and its application},
  author={Hung, Li-Chang and Liu, Hsiao-Feng and Chen, Chiun-Chuan},
  journal={Discrete and Continuous Dynamical Systems},
  year={2016}
}

@article{chen2020discrete,
  title={Discrete {N}-barrier maximum principle for a lattice dynamical system arising in competition models.},
  author={Chen, Chiun-Chuan and Hsiao, Ting-Yang and Hung, Li-Chang},
  journal={Discrete \& Continuous Dynamical Systems: Series A},
  volume={40},
  number={1},
  year={2020}
}

@article{chen2016nonexistence,
  title={NONEXISTENCE OF TRAVELING WAVE SOLUTIONS, EXACT AND SEMI-EXACT TRAVELING WAVE SOLUTIONS FOR DIFFUSIVE {L}OTKA-{V}OLTERRA SYSTEMS OF THREE COMPETING SPECIES.},
  author={Chen, Chiun-Chuan and Hung, Li-Chang},
  journal={Communications on Pure \& Applied Analysis},
  volume={15},
  number={4},
  year={2016}
}

@article{chen2016n,
  title={AN {N}-barrier maximum principle for autonomous systems of n species and its application to problems arising from population dynamics},
  author={Chen, Chiun-Chuan and Hung, Li-Chang and Lai, Chen-Chih},
  journal={Communications on Pure \& Applied Analysis},
  volume={18},
  year={2016},
}

@article{chen2016maximum,
  title={A maximum principle for diffusive {L}otka-{V}olterra systems of two competing species},
  author={Chen, Chiun-Chuan and Hung, Li-Chang},
  journal={Journal of Differential Equations},
  volume={261},
  number={8},
  pages={4573--4592},
  year={2016},
  publisher={Elsevier}
}

@article{alfaro2023lotka,
  title={{L}otka-{V}olterra competition-diffusion system: the critical competition case},
  author={Alfaro, Matthieu and Xiao, Dongyuan},
  journal={Communications in Partial Differential Equations},
  volume={48},
  number={2},
  pages={182--208},
  year={2023},
  publisher={Taylor \& Francis}
}

@article{cai2025lotka,
  title={Lotka-{V}olterra competition system with critical competition case in high-dimensional space},
  author={Cai, Jingjing and Ma, Shizhao and Xiao, Dongyuan},
  journal={Discrete and Continuous Dynamical Systems-B},
  volume={30},
  number={9},
  pages={3505--3525},
  year={2025},
  publisher={Discrete and Continuous Dynamical Systems-B}
}

@article{van1995existence,
  title={The existence of travelling plane waves in a general class of competition-diffusion systems},
  author={van Vuuren, Jan H},
  journal={IMA journal of applied mathematics},
  volume={55},
  number={2},
  pages={135--148},
  year={1995},
  publisher={Oxford University Press}
}

@article{wu2013spreading,
  title={SPREADING SPEED AND TRAVELING WAVES FOR A TWO-SPECIES WEAK COMPETITION SYSTEM WITH FREE BOUNDARY.},
  author={Wu, Chang-Hong},
  journal={Discrete \& Continuous Dynamical Systems-Series B},
  volume={18},
  number={9},
  year={2013}
}

@article{marion2023existence,
  title={Existence of pulses for monotone reaction-diffusion systems},
  author={Marion, Martine and Volpert, Vitaly},
  journal={SIAM Journal on Mathematical Analysis},
  volume={55},
  number={2},
  pages={603--627},
  year={2023},
  publisher={SIAM}
}

@article{hung2012exact,
  title={Exact traveling wave solutions for diffusive Lotka--Volterra systems of two competing species},
  author={Hung, Li-Chang},
  journal={Japan journal of industrial and applied mathematics},
  volume={29},
  number={2},
  pages={237--251},
  year={2012},
  publisher={Springer}
}
\end{center}

\end{CJK}
\end{document}